\documentclass{amsart}

\usepackage{graphicx}
\usepackage[margin=1in]{geometry} 
\usepackage{amsmath,amsthm,amssymb,amsfonts,mathrsfs}
\usepackage{tikz-cd}
\usepackage{bbm}
\usepackage{geometry}
\usepackage{hyperref}
\usetikzlibrary{matrix}

\newcommand{\GG}{Gr_G\tilde{\times}Gr_G}
\newcommand{\TT}{\tilde{\times}}
\newcommand{\UU}{S_{\mu_1}\tilde{\times} S_{\mu_2-\mu_1}}
\newtheorem{thm}{Theorem}[section]
\newtheorem{lemma}[thm]{Lemma}
\newtheorem{proposition}[thm]{Proposition}
\newtheorem{corollary}[thm]{Corollary}
\newtheorem{remark}[thm]{Remark}
\theoremstyle{definition}
\newtheorem{definition}[thm]{Definition}

\theoremstyle{remark}

\newcommand{\F}{\mathcal{F}}
\newcommand{\G}{\mathcal{G}}

\numberwithin{equation}{section}

\begin{document}

\title{The Integral Geometric Satake Equivalence in Mixed Characteristic}

\author{Jize YU}
\address{Department of Mathematics, Caltech, 1200 East California Boulevard, Pasadena, CA
91125.
}
\curraddr{}
\email{jyu3@caltech.edu}
\thanks{}

\date{}

\begin{abstract}
Let $k$ be an algebraically closed field of characteristic $p$. Denote by $W(k)$ the ring of Witt vectors of $k$. Let $F$ denote a totally ramified finite extension of $W(k)[1/p]$ and $\mathcal{O}$ the its ring of integers. For a connected reductive group scheme $G$ over $\mathcal{O}$, we study the category $P_{L^+G}(Gr_G,\Lambda)$ of $L^+G$-equivariant perverse sheaves in $\Lambda$-coefficient on the affine Grassmannian $Gr_G$ where $\Lambda=\mathbb{Z}_{\ell}$ and $\mathbb{F}_{\ell}$ and prove it is equivalent as a tensor category to the category of  finitely generated $\Lambda$-representations of the Langlands dual group of $G$.
\end{abstract}

\maketitle
\tableofcontents

\section{Introduction}

\subsection{Main Result}
 Consider an algebraically closed field $k$ of characteristic $p>0$ and denote by $W(k)$ its ring of Witt vectors. Let $F$ denote a totally ramified finite extension of $W(k)[1/p]$ and $\mathcal{O}$ the ring of integers of $F$. Let $G$ be a connected reductive group over $\mathcal{O}$ and $Gr_G$ be the Witt vector affine Grassmannian defined as in \cite{Zh2}. In this paper, we consider the category $P_{L^+G}(Gr_G,\Lambda)$ of $L^+G$-equivariant perverse sheaves in $\Lambda$-coefficient on the affine Grassmannian $Gr_G$ for $\Lambda=\mathbb{F}_{\ell}$ and $\mathbb{Z}_{\ell}$, where $\ell$ is a prime number different from $p$. We call this category the Satake category and sometimes write it as $\textnormal{Sat}_{G,\Lambda}$ for simplicity. The convolution product of sheaves equips the Satake category with a monoidal structure. Let $\hat{G}_{\Lambda}$ denote the Langlands dual group of $G$, i.e. the canonical smooth split reductive group scheme over $\Lambda$ whose root
datum is dual to that of $G$. Our main theorem is the geometric Satake equivalence in the current setting. 
\begin{thm}
There is an equivalence of monoidal categories between $P_{L^+G}(Gr_G,\Lambda)$ and the category of representations of the Langlands dual group $\hat{G}_{\Lambda}$ of $G$ on finitely generated $\Lambda$-modules.
\end{thm}
We mention that Peter Scholze has announced the same result as part of his work on the local Langlands conjecture for $p$-adic groups using a different approach.
\\

The equal characteristic counterpart of the geometric Satake equivalence  was previously achieved by the works of Drinfeld, Ginzburg, Lusztig, and Mirkovic-Vilonen(cf. \cite{BD}, \cite{Gi}, \cite{Lu1}, \cite{MV2}). Later, Zhu \cite{Zh2} considered the category of $L^+G$-equivariant perverse sheaves in $\bar{\mathbb{Q}}_{\ell}$-coefficient on the mixed characteristic affine Grassmannian $Gr_G$ and established the geometric Satake equivalence in this setting.
\\

In the equal characteristic case, the Beilinson-Drinfeld Grassmannians play a crucial role in establishing the geometric Satake equivalence. In fact, they can be used to construct the monoidal structure of the hypercohomology functor 
$$
\textnormal{H}^{*}: P_{L^+G}(Gr_G,\Lambda)\longrightarrow \textnormal{Mod}_{\Lambda}
$$
and the commutativity constraint in the Satake category by interpreting the convolution product as fusion product. In mixed characteristic, Peter Scholze's theory of diamonds allows him to construct an analogue of the Beilinson-Drinfeld Grassmannian and prove the geometric Satake equivalence in this setting in a similar way as in \cite{MV2}. Our approach of constructing the geometric Satake equivalence makes use of some ideas in \cite{Zh2}. However, our situation is different from \textit{loc.cit} and new difficulties arise. For example, the Satake category in $\bar{\mathbb{Q}}_{\ell}$-coefficient is semisimple; while, in our case, the semisimplicity of the Satake category fails. In addition, the monoidal structure of the hypercohomology functor was constructed by studying the equivariant cohomology of (convolutions of) irreducible objects in the Satake category. Nevertheless, in In our situation, the equivariant cohomology may have torsion. Thus the method in \cite{Zh2} doesn't apply to our case directly. To deal with these difficulties, we give a new approach to construct the the monoidal structure and the commutativity constraint in the Satake category and we briefly discuss our strategy as follows.
\\

We consider the hypercohomology functor $\textnormal{H}^*$ and prove that it is $\Lambda$-linear, exact, and faithful. Then we study the $\mathbb{G}_m$-action (in fact, we consider the action of the perfection of the group scheme $\mathbb{G}_m$)  on the convolution Grassmannian $Gr_G\tilde{\times}Gr_G$.  Applying the Mirkovi\'{c}-Vilonen theory for mixed characteristic affine Grassmannians established in \cite{Zh2} and Braden's hyperbolic localization functor, we can decompose the hypercohomology functor on $Gr_G\tilde{\times}Gr_G$ into a direct sum of compactly supported cohomologies. Each direct summand can be further realized as the tensor product of two compactly supported cohomologies on $Gr_G$ by the K\"{u}nneth formula. Putting these together gives the desired monoidal structure on $\textnormal{H}^{*}$ which is compatible with the monoidal structure obtained in \cite{Zh2}.
\\

 We further notice that as in the cases discussed in \cite{MV2} and \cite{Zh2}, the hypercohomology functor is representable by projective objects when restricting to full subcategories of the Satake category. In particular, these objects are isomorphic to the projective objects in the case studied in \cite{Zh2} after base change. This allows us to directly construct a $\Lambda$-algebra $B(\Lambda)$ as in \cite{MV2}. The compatibility of the monoidal structure of $
 H^*$ and the projective objects constructed in our case with those in \cite{Zh2} enable us to inherit a commutative multiplication map in $B(\Lambda)$ from that of $B(\bar{\mathbb{Q}}_{\ell})$ which is given by the commutativity constraint constructed in \textit{loc.cit}. In this way we equip $B(\Lambda)$ with the structure of a bialgebra in the sense of \cite{DM}. The general Tannakian construction (cf.\cite{MV2}) yields an equivalence of tensor categories 
 $$
P_{L^+G}(Gr_G,\Lambda)\simeq \textnormal{Rep}_{\Lambda}(\widetilde{G}_{\Lambda}),
$$
where $\widetilde{G}_{\Lambda}:=\textnormal{Spec}(B_{\Lambda})$ is an affine flat group scheme and $\textnormal{Rep}_{\Lambda}(\widetilde{G}_{\Lambda})$ denotes the category of $\Lambda$-representations of $\widetilde{G}_{\Lambda}$ which are finitely generated over $\Lambda$. 
We give two approaches identifying $\widetilde{G}_{\Lambda}$ with $\hat{G}_{\Lambda}$ and conclude the proof of the theorem.
\subsection{Organization of the paper}
We briefly discuss the organization of this paper. In \S $2$, we give a quick review the of construction of affine Grassmannians in mixed characteristic. Section $3$ is devoted to the construction of the Satake category and its monoidal structure. We study the hyperbolic localization functor on the affine Grassmannian and define the weight functors in \S $4$. In \S $5$, we prove that the weight functors are representable and study the structure of the projective objects. In \S 6, we show the hypercohomology functor $\textnormal{H}^{*}$ can be endowed with a monoidal structure. We apply the Tannakian formalism and construct the commutativity constraint in \S $7$. In \S $8$, we discuss the identification of group schemes and conclude the main theorem of this paper. 

\subsection{Notations}
In this section, we fix a few notations for later use.
\\

We fix an algebraically closed field $k$ of characteristic $p> 0$. For any $k$-algebra $R$, its ring of Witt vectors is denoted by 
$$
W(R)=\{(r_0,r_1,\cdots)\vert r_i\in R\}.
$$
We denote by $W_h(R)$ the ring of truncated Witt vectors of length $h$. For perfect $k$-algebra $R$,  we know that $W_h(R) = W(R)/p^hW(R)$.
\\

Let $\mathcal{O}_0 = W(k)$, and $F_0 = W(k)[1/p]$. We denote by $F$ a totally ramified
finite extension of $F_0$ and $\mathcal{O}$ the ring of integers of $F$. We also fix a uniformizer $\varpi$ in $\mathcal{O}$. For any $k$-algebra $R$, we use notations
$$
W_{\mathcal{O}}(R)=W(R)\otimes_{W(k)} \mathcal{O},\text{ and }W_{\mathcal{O},n}(R)=W(R)\otimes_{W(k)} \mathcal{O}/\varpi^n.
$$
We define the formal unit disk and formal punctured unite disk to be
$$
D_{F,R}:=\textnormal{Spec}W_{\mathcal{O}}(R),\text{ and } D_{F,R}^{\times}:=\textnormal{Spec}W_{\mathcal{O}}(R)[1/p]
$$
respectively.
We will assume $G$ to be a smooth affine group scheme over $\mathcal{O}$
with connected geometric fibers. Also, we denote by $\mathcal{E}_0$ the trivial $G$-torsor.
\\

In the case $G$ is a split reductive group, we will choose a Borel subgroup $B\subset G$ over $\mathcal{O}$ and a split maximal torus $T\subset B$. For the choice of $(B,T)$ we denote by $U \subset B$ the unipotent radical of $B$. We also let $\bar{T}\subset \bar{B}\subset \bar{G}$ be the fibres of $T\subset B\subset G$ at $\mathcal{O}/\varpi$, respectively. We write $\Delta(G,T)$ for the root system of $G$ with respect to $T$. We denote by $\Delta_+(G,B,T)$ the subset of positive roots determined by $B$ and $\Delta_s(G,B,T)$ the subset of simple root. Similarly, we let $\Delta^{\vee}(G,T)$ denote the coroot system of $G$ respect to $T$. We write $\Delta^{\vee}_+(G,B,T)$ for the subset of positive coroots determined by $B$ and $\Delta^{\vee}_s(G,B,T)$ the subset of simple coroot.
\\

Let $\mathbb{X}_{\bullet}$ denote the coweight lattice of $T$ and $\mathbb{X}^{\bullet}$ the weight lattice. Let $\mathbb{X}_{\bullet}^{+}$ denote the
semi-group of dominant coweights with respect to the chosen Borel. Let $2\rho\in \mathbb{X}^{\bullet}$ be the sum of all positive roots. Define the partial order $"\leq"$ on $\mathbb{X}_{\bullet}$ be such that $\lambda\leq \mu$, if and only if  $\mu-\lambda$  equals a non-negative integral linear combination of positive coroots.
For any $\mu\in\mathbb{X}_{\bullet}$, denote $\varpi^{\mu}$ by the image of $\mu$ under the composition of maps
$$
\mathbb{G}_m\rightarrow T\subset G.
$$
The dual group of $G$ (over a field of characteristic zero) is denoted by $\hat{G}$. We equip it with a dual Borel $\hat{B}$ and a maximal torus $\hat{T}$ dual to $T$.
\\

We will write $\mathbb{G}_m^{p^{-\infty}}$, the perfection of the group scheme $\mathbb{G}_m$, simply as $\mathbb{G}_m$ for convenience. 
\subsection{Acknowledgements} I would like to thank Xinwen Zhu for many helpful discussions and comments. I also thank Kari Vilonen for discussions on the geometric Satake equivalence in the equal characteristic case.

\section{Mixed Characteristic Affine Grassmannians}
In this section, we briefly review the construction of  affine Grassmannians in mixed characteristic and summarize their geometric properties which will be used later following \cite{Zh2}. Most properties appeared in this section have analogies in the equal characteristic setting and we refer to \cite{MV2} for a detailed discussion.
\\

We start this section by defining $p$-adic loop groups that are similar to their equal characteristic counterparts. Let $\mathcal{X}$ be a finite type $\mathcal{O}$-scheme. We consider the
following two presheaves on the category of affine $k$-schemes defined as follows
$$
L^+_p \mathcal{X}(R): = \mathcal{X}(W_{\mathcal{O}}(R)), \text{ and } L^h_p\mathcal{X}(R):= X(W_{\mathcal{O},h}(R)),
$$
which are represented by schemes over $k$. Their perfections are denoted by
$$
L^{+}\mathcal{X}:=(L^+_p \mathcal{X})^{p^{-\infty}}, \text{ and } L^h\mathcal{X}:=(L^h_p\mathcal{X})^{p^{-\infty}}
$$
respectively, and we call them p-adic jet spaces. 
\\

Let $X$ be an affine scheme over $F$. We define the $p$-adic loop space $LX$ of $X$ as the perfect
space by assigning a perfect $k$-algebra $R$ to the set
$$
LX(R) = X(W_{\mathcal{O}}(R)[1/p]).
$$
Now, let $\mathcal{X} = G$ be a smooth affine group scheme over $\mathcal{O}$. We write $G^{(0)} = G$ and define the $h$-th congruence group scheme of $G$ over $\mathcal{O}$ , denoted by $G^{(h)}$, as the dilatation of $G^{(h-1)}$ along the unit. The group $L^+G^{(h)}$ can be identified with $\textnormal{ker}(L^{+}G\rightarrow L^hG)$ via the natural map $G^{(h)}\rightarrow G$. Then $L^{+}G$ acts on $LG$ by multiplication on the right. We define the affine Grassmannian $Gr_G$ of $G$ to be the perfect space
$$
Gr_G := [LG/L^+G]
$$
on the category of perfect $k$-algebras.
\\

In the work of Bhatt-Scholze \cite{BS}, the functor $Gr_{G}$ is proved to be representable by an inductive limit of perfections
of projective varieties. 
\\

We recall the following proposition in \cite{Zh2} which will be used later.

\begin{proposition}
Let $\rho:G\rightarrow GL_n$ be a linear representation such that $GL_n/G$ is quasi-affine. Then $\rho$ induces a locally closed embedding $Gr_G\rightarrow Gr_{GL_n}$. If in addition, $GL_n/G$ is affine, then $Gr_G\rightarrow Gr_{GL_n}$ is in fact a closed embedding. 
\end{proposition}

Explicitly, the affine Grassmannian $Gr_G$ can be described as assigning a perfect $k$-algebra $R$ the set of
pairs $(P, \phi)$, where $P$ is an $L^+G$-torsor over $\textnormal{Spec} R$ and $\phi: P \rightarrow LG$ is an $L^+G$-equivariant morphism. It is clear from the definition that $LG\rightarrow Gr_G$ is an $L^{+}G$-torsor and $L^{+}G$ naturally acts on $Gr_G$, then we can form the twisted product which we also call the convolution product in the current setting 
$$
Gr_G\tilde{\times}Gr_{G}:=LG\times^{L^{+}G}Gr_G:=[LG\times Gr_G/L^{+}G],
$$
where $L^{+}G$ acts on $LG\times Gr_G$ diagonally as $g^{+}\cdot (g_1,g_2):=(g_1(g^{+})^{-1},g^{+}g_2)$. 
\\

As in the equal characteristic case, the affine Grassmannians can be interpreted as the moduli stack of $G$-torsors on the formal unit disk with trivialization away from the origin. More precisely, for each perfect $k$-algebra $R$, 

\begin{equation*}
Gr_G(R)=\left\{(\mathcal{E},\phi)\bigg|
\begin{aligned}
&	\mathcal{E}\rightarrow D_{F,R} \text{ is a } G\text{-torsor, and } \\
& \phi:\mathcal{E}\vert_{D_{F,R}^{*}}\simeq \mathcal{E}_{0}\vert_{D_{F,R}^{*}} \\
\end{aligned}
\right\}.
\end{equation*}
\\

Let $\mathcal{E}_1$ and $\mathcal{E}_2$ be two $G$-torsors over $D_{F,R}$, and let  $\beta:\mathcal{E}_1\vert_{D_{F,R}^{*}}\simeq \mathcal{E}_2\vert_{D_{F,R}^{*}}$
be an isomorphism. One can define the relative position $\textnormal{Inv}(\beta)$ of $\beta$ as an element in $\mathbb{X}^{+}_{\bullet}$ as in \cite{Zh2}.
\begin{definition}
For each $\mu\in \mathbb{X}_{\bullet}^{+}$, we define 
\begin{itemize}
    \item [(1)] the (spherical) Schubert variety
    $$
Gr_{\leq\mu}:=\{(\mathcal{E},\beta)\in Gr_G\vert \textnormal{Inv}(\beta)\leq \mu\},
$$ 
    \item[(2)] the Schubert cell
    $$
Gr_{\mu}:=\{(\mathcal{E},\beta)\in Gr_G\vert \textnormal{Inv}(\beta)= \mu\}.
$$ 
\end{itemize} to

\end{definition}

\begin{proposition}
\begin{itemize}
    \item [(1)]Let $\mu\in \mathbb{X}^{+}_{\bullet}$, and $\varpi^{\mu}\in Gr_G$  be the corresponding point in the affine Grassmannian. Then the map
$$ i_{\mu}:L^{+}G/(L^{+}G\cap \varpi^{\mu}L^{+}G\varpi^{-\mu})\longrightarrow LG/L^{+}G,\text{ such that } g\longmapsto g\varpi^{\mu}$$
induces an isomorphism 
$$
L^{+}G/(L^{+}G\cap \varpi^{\mu}L^{+}G\varpi^{-\mu})\simeq Gr_{\mu}.
$$

     \item[(2)]$Gr_{\mu}$ is the perfection of a quasi-projective smooth variety of dimension $(2\rho, \mu)$.
     \item[(3)] $Gr_{\leq\mu}$ is the Zariski closure of $Gr_{\mu}$ in $Gr_G$ and therefore is perfectly proper of dimension
$(2\rho,\mu)$.
\end{itemize}
\end{proposition}

The convolution Grassmannian $Gr_G\tilde{\times}Gr_{G}$ admits a moduli interpretation as follows
\begin{equation*}
Gr_G\tilde{\times}Gr_{G}(R)=\left\{(\mathcal{E}_1,\mathcal{E}_2,\beta_1,\beta_2)\bigg|
\begin{aligned}
&	\mathcal{E}_1,\mathcal{E}_2 \text{ are } G-\text{torsors on }D_{F,R}, \text{and} \\
& \beta_1:\mathcal{E}_1\vert_{D_{F,R}^*}\simeq \mathcal{E}_0\vert_{D_{F,R}^*}, \beta_2:\mathcal{E}_2\vert_{D_{F,R}^*}\simeq \mathcal{E}_1\vert_{D_{F,R}^*} \\
\end{aligned}
\right\}.
\end{equation*}
Via this interpretation, we define the convolution product as in the equal characteristic case 
$$
m:Gr_G\tilde{\times}Gr_{G}\longrightarrow Gr_G, 
$$
such that 
$$
(\mathcal{E}_1,\mathcal{E}_2,\beta_1,\beta_2)\longmapsto (\mathcal{E}_2,\beta_1\beta_2).
$$
Note there is also the natural projection morphism 
$$
pr_1:Gr_G\tilde{\times}Gr_{G}\longrightarrow Gr_G,
$$
such that 
$$
(\mathcal{E}_1,\mathcal{E}_2,\beta_1,\beta_2)\longmapsto (\mathcal{E}_1,\beta_1).
$$
It is clear to see $(pr_1,m):Gr_G\tilde{\times}Gr_{G}\simeq Gr_G\times Gr_G$ is an isomorphism.
\\

One can define the $n$-fold convolution Grassmannian $Gr_G\tilde{\times}\cdots\tilde{\times}Gr_G$ in a similar manner as follows
\begin{equation*}
Gr_G\tilde{\times}\cdots\tilde{\times}Gr_G:=\left\{(\mathcal{E}_i,\beta_i)\bigg|
\begin{aligned}
&	\mathcal{E}_i\text{ is a }G\text{-torsor over }D_{F,R}, \text{and} \\
& \beta_i:\mathcal{E}_i\vert_{D_{F,R}^*}\simeq \mathcal{E}_{i-1}\vert_{D_{F,R}^*} \\
\end{aligned}
\right\}.
\end{equation*}
We define the morphism 
$$
m_i:Gr_G\tilde{\times}\cdots\tilde{\times}Gr_G\longrightarrow Gr_G
$$
such that 
$$
(\mathcal{E}_i,\beta_i)\longmapsto (\mathcal{E}_i,\beta_1\beta_2\cdots \beta_i:\mathcal{E}_i\vert_{D_{F,R}^*}\simeq \mathcal{E}_0\vert_{D_{F,R}^*}).
$$
As for the $2$-fold convolution Grassmannian, we have an isomorphism
$$
(m_1,m_2,\cdots,m_n):Gr_G\tilde{\times}\cdots\tilde{\times}Gr_G\simeq Gr_G\times\cdots\times Gr_G.
$$
We also call the map $m_n$ the convolution map.
\\

Given a sequence of dominant coweights $\mu_{\bullet}=(\mu_1,\cdots,\mu_n)$ of $G$, we define the following closed subspace of $Gr_G\tilde{\times}\cdots\tilde{\times}Gr_G $,
$$
Gr_{\leq\mu_{\bullet}}:=Gr_{\leq\mu_1}\tilde{\times}\cdots\tilde{\times} Gr_{\leq\mu_n}:=\{(\mathcal{E}_i,\beta_i)\in Gr_G\tilde{\times}\cdots\tilde{\times}Gr_G\vert \textnormal{Inv}(\beta_i)\leq\mu_i\}.
$$

As in \cite{Zh2}, we let $|\mu_{\bullet}|:=\sum\mu_i$. Then the convolution map induces the following morphism
$$
m:Gr_{\leq\mu_{\bullet}}\longrightarrow Gr_{\leq |\mu_{\bullet}|},
$$
such that
$$
(\mathcal{E}_i,\beta_i)\longmapsto (\mathcal{E}_n,\beta_1\cdots\beta_n).
$$
Replace $Gr_{\leq\mu_i}$ by $Gr_{\mu_i}$, we can similarly define $Gr_{\mu_{\bullet}}:=Gr_{\mu_1}\tilde{\times}\cdots\tilde{\times} Gr_{\mu_n}$. By Proposition $2.3$, we have
\begin{equation}
Gr_{\leq\mu_{\bullet}}=\cup_{\mu_{\bullet}'\leq \mu_{\bullet}}Gr_{\mu_{\bullet}'},
\end{equation}
where $\mu_{\bullet}'\leq \mu_{\bullet}$ means $\mu_i'\leq\mu_i$ for each $i$. This gives a stratification of $Gr_{\mu_{\bullet}}$
\\

\section{The Satake Category}

In this section, we first define the Satake category $\textnormal{Sat}_{G,\Lambda}$ as the category of $L^{+}G$-equivariant $\Lambda$-coefficient perverse sheaves on $Gr_G$. Then we define the convolution map which enables us to equip the Satake category with a monoidal structure. 
\\

Recall that $Gr_G$ can be written as a limit of $L^+G$-invariant closed subsets 
$$
Gr_G=\lim_{\rightarrow\mu}Gr_{\leq\mu}
$$
For each $\mu\in \mathbb{X}_{\bullet}^{+}$, the Schubert variety is perfectly proper, and the action of $L^+G$ on $Gr_{\leq\mu}$ factors through some $L^hG$ which is perfectly of finite type. Therefore, it makes sense to define the category of $L^{+}G$-equivariant perverse sheaves on $Gr_{\leq\mu}$ as in \cite[\S{2.1.1}]{Zh2}\footnote{It can be defined by a different but equivalent approach by considering the equivariant derived category. For details, see \cite{BL1}.}, which we denote by $P_{L^+G}(Gr_{\leq\mu},\Lambda)$. Then we define the Satake category as 
$$
P_{L^+G}(Gr_{G},\Lambda):=\lim_{\rightarrow\mu}P_{L^+G}(Gr_{\leq\mu},\Lambda).
$$
We denote by $\textnormal{IC}_{\mu}$ for each $\mu\in \mathbb{X}_{\bullet}^{+}$ the intersection cohomology sheaf on $Gr_{\leq\mu}$. Its restriction to each open strata $Gr_{\nu}$ is constant and in particular, $\textnormal{IC}_{\mu}\vert_{Gr_{\mu}}\simeq \Lambda[(2\rho,\mu)]$.
\\

With the above preparation, we can define the monoidal structure in $\textnormal{Sat}_{G,\Lambda}$ by Lusztig's convolution of sheaves as in the equal characteristic counterpart.
Consider the following diagram 
$$
Gr_G\times Gr_G\stackrel{p}{\longleftarrow}LG\times Gr_G\stackrel{q}{\longrightarrow} Gr_G\tilde{\times} Gr_G\stackrel{m}{\longrightarrow} Gr_G,
$$
where $p$ and $q$ are projection maps. We define for any $\mathcal{A}_1,\mathcal{A}_2\in P_{L^+G}(Gr_G,\Lambda)$, 
$$\mathcal{A}\star \mathcal{A}_2:=Rm_!(\mathcal{A}_1\tilde{\boxtimes}\mathcal{A}_2),$$
where $\mathcal{A}_1\tilde{\boxtimes}\mathcal{A}_2\in P_{L^+G}(Gr_G\tilde{\times} Gr_G,\Lambda)$ is the unique sheaf such that 
$$
q^*(\mathcal{A}_1\tilde{\boxtimes}\mathcal{A}_2)\simeq p^*({^{p}\textnormal{H}^{0}}(\mathcal{A}_1\boxtimes\mathcal{A}_2)).
$$ 
Unlike the construction in $P_{L^+G}(Gr_G,\bar{\mathbb{Q}}_{\ell})$, we emphasize that taking the $0$-th perverse cohomology in the above definition is necessary. This is because when we work with $\mathbb{Z}_{\ell}$-sheaves, the sheaf $\mathcal{A}_1\boxtimes\mathcal{A}_2$ may not be perverse. In fact, $\mathcal{A}_1\boxtimes\mathcal{A}_2$ is perverse if one of $\textnormal{H}^{*}(\mathcal{A}_i)$ is a flat $\mathbb{Z}_{\ell}$-module. For more details, please refer to \cite[Lemma 4.1]{MV2}.  
\begin{proposition}
For any $\mathcal{A}_1,\mathcal{A}_2\in P_{L^+G}(Gr_G,\Lambda)$, the convolution product $\mathcal{A}_1\star\mathcal{A}_2$ is perverse.
\end{proposition}
\begin{proof}
Note by \cite[Proposiion 2.3]{Zh2}, the convolution morphism $m$ is a stratified semi-small morphism with respect to the stratification $(2.1)$. Then the proposition follows from \cite[Lemma 4.3]{MV2}
\end{proof}

We can also define the $n$-fold convolution production in $\textnormal{Sat}_{G,\Lambda}$ 
$$
\mathcal{A}_1\star \cdots\star\mathcal{A}_n:=Rm_!(\mathcal{A}_1\tilde{\boxtimes}\cdots\tilde{\boxtimes}\mathcal{A}_n).
$$
where $\mathcal{A}_1\tilde{\boxtimes}\cdots\tilde{\boxtimes}\mathcal{A}_n$ is defined in a similar way as $\mathcal{A}_1\tilde{\boxtimes}\mathcal{A}_2$.
By considering the following isomorphism
$$
{^{p}\textnormal{H}^{0}}(\mathcal{A}_1\boxtimes   ({^{p}\textnormal{H}^{0}}(\mathcal{A}_2 \boxtimes \mathcal{A}_3))\simeq {^{p}\textnormal{H}^{0}}(\mathcal{A}_1\boxtimes\mathcal{A}_2\boxtimes\mathcal{A}_3)\simeq {^{p}\textnormal{H}^{0}}({^{p}\textnormal{H}^{0}}(\mathcal{A}_1\boxtimes\mathcal{A}_2)\boxtimes\mathcal{A}_3),
$$
we conclude that the convolution product is associative
$$
(\mathcal{A}_1\star\mathcal{A}_2)\star\mathcal{A}_3\simeq Rm_!(\mathcal{A}_1\tilde{\boxtimes}\mathcal{A}_2\tilde{\boxtimes}\mathcal{A}_3)\simeq \mathcal{A}_1\star(\mathcal{A}_2\star\mathcal{A}_3).
$$
Thus the category $(\textnormal{Sat}_{G,\Lambda},\star)$ is a monoidal category.

\section{Semi-Infinite Orbits and Weight Functors}

In this section, we review the construction and geometry of semi-infinite orbits of $Gr_G$. By studying the $\mathbb{G}_m$-action on the affine Grassmannian $Gr_G$, we realize the semi-infinite orbits as the attracting loci of the $\mathbb{G}_m$-action in the sense of \cite{DG}. We also define the weight functors and relate them to the hyperbolic localization functors and study their properties. 
\\

Follow notions introduced in \S 1.3 and let $U$ be the unipotent radical of $G$. Since $U\backslash G$ is quasi-affine, recall Proposition $2.1$ we know that that $Gr_U\hookrightarrow Gr_G$
is a locally closed embedding. For any $\lambda\in \mathbb{X}_{\bullet}$ define 
$$
S_{\lambda}:=LU\varpi^{\lambda}
$$
to be the orbit of $\varpi^{\lambda}$ under the $LU$-action. Then $S_{\lambda}=\varpi^{\lambda}Gr_U$ and therefore is locally closed in $Gr_G$ via the emebdding $Gr_U\hookrightarrow Gr_G$. By the Iwasawa decomposition for $p$-adic groups, we know that 
$$
Gr_G=\cup_{\lambda\in\mathbb{X}_{\bullet}}S_{\lambda}.
$$
Similarly, consider the opposite Borel $B^{-}$ and let $U^{-}$ be its unipotent radical. Then we define the opposite semi-infinite orbits 
$$
S_{\lambda}^{-}:=LU^{-}\varpi^{\lambda}.
$$
for $\lambda\in\mathbb{X}_{\bullet}$.
\\

We also recall the following closure relations as in \cite[Proposition 2.5]{Zh2}  (the equal characteristic analogue of this statement is proved in \cite[Proposition 3.1]{MV2} )
\begin{proposition}
Let $\lambda\in\mathbb{X}_{\bullet}$, then $S_{\leq\lambda}:=\bar{S_{\lambda}}=\cup_{\lambda'\leq\lambda}S_{\lambda'}$ and $S_{\leq\lambda}^{-}:=\bar{S_{\lambda}^{-}}=\cup_{\lambda'\leq\lambda}S_{\lambda'}^{-}$.
\end{proposition}
Similar to the equal characteristic situation (cf. \cite{MV2} $(3.16)$, $(3.17)$), the semi-infinite orbits may be interpreted as the attracting loci of certain torus action which we describe here. 
\\

Let $2\rho^{\vee}$ be the sum of all positive coroots of $G$ with respect to $B$ and regard it as a cocharacter of $G$. The projection map $L^{+}_p\mathbb{G}_m\rightarrow \mathbb{G}_m$ admits a unique section $\mathbb{G}_m\rightarrow L^{+}_p\mathbb{G}_m$ which identifies $\mathbb{G}_m$ as the maximal torus of $L^{+}_p\mathbb{G}_m$. This section allows us to define a cocharacter 

$$
\mathbb{G}_m\longrightarrow L^{+}\mathbb{G}_m\stackrel{L^{+}(2\rho^{\vee})}{\longrightarrow} L^{+}T\subset L^{+}G.
$$
Then the $\mathbb{G}_m$-action on $Gr_G$ is induced by the action of $L^+G$ on $Gr_G$. Under this action by $\mathbb{G}_m$, the set of fixed points are precisely $R:=\{\varpi^{\lambda}\vert \lambda\in \mathbb{X}_{\bullet}\}$. The attracting loci of this action are semi-infinite orbits i.e.
$$
S_{\lambda}=\{g\in Gr_G\vert \lim_{t\rightarrow 0} L^+(2\rho^{\vee}(t))\cdot(g)=\varpi^{\lambda}\text{ for }t\in \mathbb{G}_m\}.
$$
The repelling loci are the opposite semi-infinite orbits
$$
S_{\lambda}^{-}=\{g\in Gr_G\vert \lim_{t\rightarrow \infty} L^+(2\rho^{\vee}(t))\cdot(g)=\varpi^{\lambda}\text{ for }t\in \mathbb{G}_m\}.
$$
Applying the reduction of structure group to the $L^{+}G$-torsor $LG\rightarrow Gr_G$ to $S_{\mu}$, we obtain an $L^{+}U$-torsor $LU\rightarrow S_{\mu}$. Then we can form the twisted product of semi-infinite orbits $S_{\mu_1}\tilde{\times} S_{\mu_2}\tilde{\times}\cdots \tilde{\times}S_{\mu_n}$. Let $\mu_{\bullet}=(\mu_1,\cdots,\mu_n)$ be a sequence of (not necessarily dominant) coweights of $G$. We define
$$
S_{\mu_{\bullet}}:=S_{\mu_1}\tilde{\times} S_{\mu_2}\tilde{\times}\cdots \tilde{\times}S_{\mu_n}\subset Gr_G\tilde{\times}Gr_G\tilde{\times}\cdots\tilde{\times}Gr_G.
$$
The morphism 
$$
m:S_{\mu_{\bullet}}\longrightarrow S_{\mu_1}\times S_{\mu_1+\mu_2}\times\cdots\times S_{|\mu_{\bullet}|}
$$
given by 
$$
(\varpi^{\mu_1}x_1,\varpi^{\mu_2}x_2,\cdots \varpi^{\mu_1n}x_n) \longmapsto (\varpi^{\mu_1}x_1, \varpi^{\mu_1+\mu_2}(\varpi^{-\mu_2}x_1\varpi^{\mu_2}x_2),\cdots, \varpi^{|\mu_{\bullet}|}(\varpi^{-|\mu_{\bullet}|+\mu_1} x_1\cdots \varpi^{\mu_n}x_n))
$$
is an isomorphism.
The morphism $m$ fits into the following commutative diagram 
$$
\begin{tikzcd}[row sep=huge]
	S_{\leq\mu_0} \arrow[rr,"m"] \arrow[d,hook] & & S_{\mu_1}\times S_{\mu_2}\times \cdots \times S_{\mu_n} \arrow[d,hook]\\
	Gr_G\tilde{\times}Gr_G\tilde{\times}\cdots\tilde{\times}Gr_G \arrow[rr,"(m_1{,}...{,}m_n)"] & & Gr_G\times Gr_G\times\cdots\times Gr_G 	
\end{tikzcd}
$$
We also note there is a canonical isomorphism 
\begin{equation}
(S_{\nu_1}\cap Gr_{\leq\mu_1})\tilde{\times} (S_{\nu_2}\cap Gr_{\leq\mu_2})\tilde{\times}\cdots\tilde{\times}(S_{\nu_n}\cap Gr_{\leq\mu_n})\simeq S_{\nu_{\bullet}}\cap Gr_{\leq\mu_{\bullet}}.
\end{equation}

Recall that if $X$ is a scheme and $i:Y\hookrightarrow X$ is an inclusion of a locally closed subscheme, then for any $\mathcal{F}\in D_c^b(X,\Lambda)$, the local cohomology group is defined as $\textnormal{H}_Y^k(X,\mathcal{F}):=\textnormal{H}^k(Y,i^!\mathcal{F})$.
\begin{proposition}
For any $\mathcal{F}\in P_{L^+G}(Gr_G,\Lambda)$, there is an isomorphism 
$$
\textnormal{H}_c^{k}(S_{\mu},\mathcal{F})\simeq \textnormal{H}_{S_{\mu}^{-}}^{k}(\mathcal{F}),
$$
and both sides vanish if $k\neq (2\rho,\mu)$.
\end{proposition}
\begin{proof}
The proof is similar to the equal characteristic case (cf \cite[Theorem 3.5]{MV2})as the dimension estimation of the intersections of the semi-infinite orbits and Schubert varieties are established in \cite[Corollary 2.8]{Zh2}. Since $\mathcal{F}$ is perverse, then for any $\nu\in \mathbb{X}_{\bullet}^{+}$, we know that $\mathcal{F}\vert_{Gr_{\nu}} \in D^{\leq-\dim(Gr_{\nu})}= D^{\leq-(2\rho,\nu)}$.
By \cite[Corollary 2.8]{Zh2}, we know that $\textnormal{H}^{k}_c(S_{\mu}\cap Gr_{\leq\nu},\mathcal{F})=0$ if $k>2\dim (S_{\mu}\cap Gr_{\leq\nu})=(2\rho,\mu+\nu)$. Filtering $Gr_G$ by $Gr_{\leq\mu}$, we apply a d\'evissage argument and conclude that
$$
\textnormal{H}_c^{k}(S_{\mu},\mathcal{F})=0 \text{ if } k>(2\rho,\mu)
.$$
An analogous argument proves that 
$$
\textnormal{H}_{S_{\mu}^{-}}^{k}(\mathcal{F})=0
$$
if $k<(2\rho,\mu)$.
\\

Now by regarding $S_{\mu}$ and $S_{\mu}^{-}$ as the attracting and repelling loci of the $\mathbb{G}_m$-action, we apply the hyperbolic localization as in \cite{DG} and obtain  
$$
\textnormal{H}_c^{k}(S_{\mu},\mathcal{F})\simeq \textnormal{H}_{S_{\mu}^{-}}^{k}(\mathcal{F}).
$$
The proposition is thus proved.
\end{proof}
Let $\textnormal{Mod}_{\Lambda}$ denote the category of finitely generated $\Lambda$-modules and $\textnormal{Mod}(\mathbb{X}_{\bullet})$ denote the category of $\mathbb{X}_{\bullet}$-graded finitely generated $\Lambda$-modules.
\begin{definition}
\begin{itemize}For any $\mu\in \mathbb{X}_{\bullet}$, we define
    \item [(1)]the weight functor
$$
\textnormal{CT}_{\mu}:P_{L^+G}(Gr_G,\Lambda)\longrightarrow \textnormal{Mod}(\mathbb{X}_{\bullet}),
$$
by
$$
\textnormal{CT}_{\mu}(\mathcal{F}):=\textnormal{H}_c^{(2\rho,\mu)}(S_{\mu},\mathcal{F}).
$$
\item[(2)]the total weight functor 
$$
\textnormal{CT}:=\bigoplus_{\mu} \textnormal{CT}_{\mu}:P_{L^+G}(Gr_G,\Lambda)\longrightarrow \textnormal{Mod}(\mathbb{X}_{\bullet}),
$$
by
$$
\textnormal{CT}(\mathcal{F}):=\bigoplus_{\mu}\textnormal{CT}_{\mu}(\mathcal{F}):=\bigoplus_{\mu}\textnormal{H}_c^{(2\rho,\mu)}(S_{\mu},\mathcal{F}).
$$
\end{itemize}
\end{definition}
We denote by $F$ the forgetful functor from $\textnormal{Mod}(\mathbb{X}_{\bullet})$ to  $\textnormal{Mod}_{\Lambda}$. 
\begin{proposition}
There is a canonical isomorphism of functors 
$$
\textnormal{H}^{*}(Gr_G,\bullet)\cong F \circ\textnormal{CT}:P_{L^+G}(Gr_G,\Lambda)\longrightarrow \textnormal{Mod}_{\Lambda}.
$$
In addition, both functors are exact and faithful.
\end{proposition}
\begin{proof}
By the definition of the semi-infinite orbits and the Iwasawa decomposition, we obtain two stratifications of $Gr_G$ by $\{S_{\mu}\vert \mu\in \mathbb{X}_{\bullet}\}$ and $\{S_{\mu}^{-}\vert \mu\in \mathbb{X}_{\bullet}\}$, respectively. The first stratification induces a spectral sequence with $E_1$-terms $\textnormal{H}_c^{k}(S_{\mu},\mathcal{F})$ and abutment $\textnormal{H}^{*}(Gr_G,\mathcal{F})$. This spectral sequence degenerates at the $E_1$-page by Proposition $4.2$. Thus, there is a filtration of $ \textnormal{H}^{*}(Gr_G,\mathcal{F})$ indexed by $(\mathbb{X}_{\bullet}, \leq)$ defined as
$$
\textnormal{Fil}_{\geq\mu}\textnormal{H}^{*}(Gr_G,\mathcal{F}):=\textnormal{ker}(\textnormal{H}^{*}(Gr_G,\mathcal{F})\longrightarrow \textnormal{H}^{*}(S_{<\mu},\mathcal{F})),
$$
where $S_{<\mu}:=\cup_{\mu'<\mu}S_{\mu'}$. Direct computation yields that the associated graded of the above filtration is $\bigoplus_{\mu}\textnormal{H}_c^{(2\rho,\mu)}(S_{\mu},\mathcal{F})$.
\\

Consider the second stratification of $Gr_G$. It also induces a filtration of $\textnormal{H}^{*}(Gr_G,\mathcal{F})$ as
$$
\textnormal{Fil}'_{<\mu}\textnormal{H}^{k}(Gr_G,\mathcal{F}):=\textnormal{Im}(\textnormal{H}^{*}_{S_{\leq\mu}^{-}}(\mathcal{F})\longrightarrow \textnormal{H}^{*}(Gr_G,\mathcal{F})),
$$
where $S_{<\mu}^{-}:=\cup_{\mu'<\mu}S_{\mu'}^{-}$. 
\\

Now, by the previous proposition, the two filtrations are complementary to each other and together define the decomposition $\textnormal{H}^{*}(Gr_G,\bullet)\simeq \bigoplus_{\mu}\textnormal{H}_c^{(2\rho,\mu)}(S_{\mu},\bullet)$.
\\

Next, we prove that the total weight functor $\textnormal{CT}$ is exact. To do so, it suffices to show the weight functor $\textnormal{CT}_{\mu}$ is exact for each $\mu\in\mathbb{X}_{\bullet}$. Let 
$$
0\longrightarrow \mathcal{F}_1\longrightarrow \mathcal{F}_2 \longrightarrow \mathcal{F}_3\longrightarrow 0
$$ 
be an exact sequence in $P_{L^+G}(Gr_G,\Lambda)$. It is given by a distinguished triangle 
$$
\mathcal{F}_1\longrightarrow \mathcal{F}_2\longrightarrow\mathcal{F}_3\stackrel{+1}{\longrightarrow}
$$
in $D_{c}^b(Gr_G,\Lambda)$.
We thus have a long exact sequence of cohomology 
$$
\cdots \longrightarrow \textnormal{H}_c^k(S_{\mu},\mathcal{F}_1)\longrightarrow \textnormal{H}_c^k(S_{\mu},\mathcal{F}_2)\longrightarrow \textnormal{H}^k_c(S_{\mu},\mathcal{F}_3) \longrightarrow \textnormal{H}^k_c(S_{\mu},\mathcal{F}_3)\longrightarrow \textnormal{H}^{k+1}_c(S_{\mu},\mathcal{F}_1)\longrightarrow\cdots .
$$
Then Proposition $4.2$ gives the desired exact sequence 
$$
0\longrightarrow \textnormal{CT}_{\mu}(\mathcal{F}_1)\longrightarrow \textnormal{CT}_{\mu}(\mathcal{F}_2) \longrightarrow \textnormal{CT}_{\mu}(\mathcal{F}_3)\longrightarrow 0.
$$ 

We conclude the proof by showing that $\textnormal{CT}$ is faithful. Since $\textnormal{CT}$ is exact, it suffices to prove that $\textnormal{CT}$ maps non-zero objects to non-zero objects. Let $\mathcal{F}\in \textnormal{Sat}_{G,\Lambda}$ be a nonzero object. Then $\textnormal{supp} (\mathcal{F})$ is a finite union of Schubert cells $Gr_{\nu}$. Choose $\nu$ to be maximal for this property. Then $\mathcal{F}\vert_{Gr_{\nu}}\simeq \underline{\Lambda}^{\oplus n}[(2\rho,\nu)]$ for some natural number $n$ and it follows that $\textnormal{CT}_{\nu}(\mathcal{F})\neq 0$. Thus the functor $\textnormal{H}^{*}$ is faithful.
\end{proof}
\begin{remark}
The weight functor is in fact independent of the choice of the maximal torus $T$. The proof for this is analogous to the equal characteristic case (cf.\cite[Theorem 3.6]{MV2})
\end{remark}
We note the analogue of \cite[Corollary 2.9]{Zh2} also holds in our setting. In particular, $\textnormal{H}^{*}(\textnormal{IC}_{\mu})$ is a free $\Lambda$-module for any $\mu\in\mathbb{X}_{\bullet}^{+}$.
\\

We end this section by proving a weaker statement of \cite[Proposition 2.1]{MV2} which will be used in the process of identification of group schemes in \S 8. 
\begin{lemma}
There is a natural equivalence of tensor categories 
$$
\alpha: P_{L^+G}(Gr_G,\Lambda)\cong P_{L^+(G/Z)} (Gr_G,\Lambda),
$$
where $Z$ is the center of $G$.
\end{lemma}
\begin{proof}
We first note that the category $P_{L^+(G/Z)} (Gr_G,\Lambda)$ can be identified as a full subcategory of $ P_{L^+G}(Gr_G,\Lambda)$. Let $X\subset Gr_G$ be a finite union of $L^+G$-orbits. Since $L^+Z$ acts on $Gr_G$ trivially, the action of $L^+G$ on $Gr_G$ factors through the quotient $L^+(G/Z)$. In other words, the following diagram commutes 
$$
\begin{tikzcd}[row sep=huge]
L^+G\times X \arrow[r,"a_1"] \arrow[d,"q"] & X \\
L^+(G/Z)\times X \arrow[ur,"a_2"]
\end{tikzcd}
$$
where $a_1$ and $a_2$ are the action maps and $q$ is the natural projection map. In addition, the following diagram is clearly commutative
$$
\begin{tikzcd}[row sep=huge]
L^+G\times X \arrow[r,"p_1"] \arrow[d,"q"] & X \\
L^+(G/Z)\times X \arrow[ur,"p_2"]
\end{tikzcd}
$$
It follows that any $\F\in P_{L^+(G/Z)}(Gr_G,\Lambda)$, $\F$ is $L^+G$-equivariant by checking the definition directly.
\\

Thus it suffices to prove reverse direction. We prove by induction on the number of $L^+G$-orbits in $X$ as in the proof of \cite[Proposition A.1]{MV2}. First, we assume $X$ contains exactly one $L^+G$-orbit. Write $X=Gr_{\mu}$ for some $\mu\in\mathbb{X}_{\bullet}^+$. Recall Proposition $2.3.(1)$ and \cite[1.4.4]{Zh2}, there is a natural projection with fibres isomorphic to the perfection of affine spaces
$$
\pi_{\mu}:Gr_{\mu}\simeq L^+G/(L^+G\cap \varpi^{\mu}L^+G\varpi^{-\mu})\longrightarrow (\bar{G}/\bar{P}_{\mu})^{p^{-\infty}}
$$
$$
 (gt^{\mu} \textnormal{ mod} L^+G) \longmapsto (\bar{g}\textnormal{ mod} \bar{P}_{\mu}^{p^{-\infty}})
$$
where $P_{\mu}$ denotes the parabolic subgroup of $G$ generated by
 the root subgroups $U_{\alpha}$ of $G$ corresponding to those roots $\alpha$ satisfying $\big<\alpha,\mu\big>\leq 0$ and $\bar{P}_{\mu}$ denotes the fibre of $P_{\mu}$ at $\mathcal{O}/\varpi$. Assume that $L^+G$ acts on $(\bar{G}/\bar{P}_{\mu})^{p^{-\infty}}$ by a finite type quotient $L^nG$. Since the stabilizer of this action of $L^+G$ is connected, we have a canonical equivalence of categories (cf.\cite[A.3.4]{Zh2}) $P_{L^+G}((\bar{G}/\bar{P}_{\mu})^{p^{-\infty}},\Lambda)\simeq P_{L^nG}((\bar{G}/\bar{P}_{\mu})^{p^{-\infty}},\Lambda)$. Finally, we note $P_{L^nG}((\bar{G}/\bar{P}_{\mu})^{p^{-\infty}},\Lambda)$ is equivalent to $\textnormal{Mod}_{\Lambda}(\mathbb{B}L^nG)$, we conclude that $P_{L^+G}(Gr_{\mu},\Lambda)\simeq \textnormal{Mod}_{\Lambda}(\mathbb{B}L^nG)$. A completely similar argument implies that $P_{L^+(G/Z)}(Gr_{\mu},\Lambda)\simeq \textnormal{Mod}_{\Lambda}(\mathbb{B}L^nG)$ which concludes the proof in the case $X=Gr_{\mu}$.
\\

Now we treat the general $X$. Choose $\mu\in\mathbb{X}_{\bullet}^+$ such that $Gr_{\mu}\subset X$ is a closed subspace, and let $U:=X\backslash Gr_{\mu}$. By induction hypothesis, we know that $P_{L^+G}(U,\Lambda)$ is equivalent to $P_{L^+(G/Z)}(U,\Lambda)$. Denote by $i:Gr_{\mu}\hookrightarrow X$ and $j:U\hookrightarrow X$ the closed and open embeddings, respectively. Let $\widetilde{Gr}_{\mu}:=L^+(G/Z)\times Gr_{\mu}$, $\widetilde{X}:=L^+(G/Z)\times X$, and $\widetilde{U}:=L^+(G/Z)\times U$. Denote by $\tilde{j}:\widetilde{U}\hookrightarrow \widetilde{X}$ the open embedding. The stratification on $X$ induces a stratification on $\widetilde{X}$ which has strata equal to products of $L^+(G/Z)$ with strata in $X$. Restricting to $\widetilde{U}$, we get a stratification of $\widetilde{U}$. Considering the action of $L^+G$ on $\widetilde{X}$ and $\widetilde{U}$ by left multiplication on the second factor, we can define categories $P_{L^+G}(\widetilde{X},\Lambda)$ and $P_{L^+G}(\widetilde{U},\Lambda)$. Define the functor
$$
\widetilde{\textnormal{CT}}_{\mu}:P_{L^+G}(\widetilde{X},\Lambda)\longrightarrow \textnormal{Loc}_{\Lambda}(L^+(G/Z)) 
$$
$$
\widetilde{\textnormal{CT}}_{\mu}(\F):=\mathcal{H}^{(2\rho,\mu)+\dim L^+(G/Z)}(\pi_!\tilde{i}^*(\F))
$$
for any $\F\in P_{L^+G}(\widetilde{X},\Lambda)$, where $\widetilde{i}:L^+(G/Z)\times (S_{\mu}\cap X)\hookrightarrow \widetilde{X}$ is the locally closed embedding, $\pi:L^+(G/Z)\times (S_{\mu}\cap X)\rightarrow L^+(G/Z)$ is the natural projection, and $\textnormal{Loc}_{\Lambda}(L^+(G/Z))$ denotes the category of $\Lambda$-local systems on $L^+(G/Z)$. A completely similar argument as in Proposition $4.2$ shows that $\widetilde{\textnormal{CT}}_{\mu}$ is an exact functor. Let $\widetilde{F}_1:=\widetilde{\textnormal{CT}}_{\mu}\circ {^{p}\tilde{j}_!}, \widetilde{F}_2:\widetilde{\textnormal{CT}}_{\mu}\circ {^{p}j_*}:P_{L^+G}(\widetilde{U},\Lambda)\rightarrow \textnormal{Loc}_{\Lambda}(L^+(G/Z))$. Finally let $\widetilde{T}:=\widetilde{\textnormal{CT}}_{\mu}({^{p}\tilde{j}_!}\rightarrow {^{p}\tilde{j}_*})$. Then as in \cite[Appendix A]{MV2}, we get an equivalence of abelian categories
$$
\widetilde{E}:P_{L^+G}(\widetilde{X},\Lambda)\simeq \mathcal{C}(\widetilde{F}_1,\widetilde{F}_2,\widetilde{T})
$$
where the second category in the above is defined in \textit{loc.cit}. Note any $\F\in P_{L^+(G/Z)}(\widetilde{X},\Lambda)$ is $\mathbb{G}_m$-equivariant, the same argument in \cite[Proposition A.1]{MV2} applies here and gives 
$$
\widetilde{E}(a_2^*\F)\simeq \widetilde{E}(p_2^*\F).
$$
Then we deduce an isomorphism $a_2^*\F\simeq p_2^*\F$ and the lemma is thus proved.

\end{proof}

\section{Representability of Weight Functors and the Structure of Representing Objects}
In section \S 4, we constructed the weight functors and the total weight functor 
$$
\textnormal{CT}_{\mu},\textnormal{CT} :P_{L^+G}(Gr_G,\Lambda)\rightarrow \textnormal{Mod}_{\Lambda}.
$$
We will prove in this section that both functors are (pro)representable, so that we can apply the (generalized) Deligne and Milne's Tannakian formalism as in \cite[\S 11]{MV2}. In the following, we will recall the induction functor (cf .\cite{MV1}) to explicitly construct the representing object of each weight functor and use the representability of the total weight functor to prove that the Satake category has enough projective objects. At the end of this section, we give a few propositions of the representing objects which will be used later when we apply the (generalized) Tannakian formalism.
\\

Let $Z\subset Gr_G$ be a closed subspace which is a union of finitely many $L^+G$-orbits. Choose $n\in\mathbb{Z}$ large enough so that $L^+G$ acts on $Z$ via the quotient $L^nG$. Let $\nu\in \mathbb{X}_{\bullet}$. As in \cite[\S 9]{MV2}, we consider the following commutative diagram 
$$
\begin{tikzcd}[row sep=huge]
S_{\nu}^{-}\cap Z \arrow[d,hook,"i"] & L^nG\times (S_{\nu}^{-}\cap Z) \arrow[l] \arrow[r,"\Tilde{a}"] \arrow[d,hook] & Z \arrow[d,equal]\\
Z & L^nG\times Z \arrow[l,"p"]\arrow[r,"a"] & Z
\end{tikzcd}
$$
where $i$ is the locally closed embedding, $a$ and $\Tilde{a}$ are the action maps, and $p$ is the projection map. Then we define 
$$
P_Z(\nu,\Lambda) :={^p\textnormal{H}}^{0}(a_!p^!i_!\underline{\Lambda}_{S_{\nu}^{-}\cap Z}[-(2\rho,\nu)])
$$
The following two results are analogues of the equal characteristic counterparts and can be proved exactly in the same manner. We omit the proofs and refer readers to \cite[Proposition 9.1, Corollary 9.2]{MV2} for details.                     
\begin{proposition}
The restriction of the weight functor $\textnormal{CT}_{\nu}$ to $P_{L^+G}(Z,\Lambda)$ is represented by the projective object $P_Z(\nu,\Lambda)$ in $P_{L^+G}(Z,\Lambda)$. 
\end{proposition}

\begin{corollary}
The category $P_{L^+G}(Z,\Lambda)$ has enough projectives.
\end{corollary}
Let $P_Z(\Lambda):=\oplus_{\nu}P_Z(\nu,\Lambda)$. We note the following mixed characteristic analogues of results of the projective objects in the equal characteristic (cf \cite[Proposition 10.1]{MV2}) hold in our setting.
\begin{proposition}
\begin{itemize}
    \item [(1)]Let $Y \subset Z$ be a closed subset which is a union of $L^+G$-orbits. Then
    $$
    P_Y(\Lambda)={^p\textnormal{H}}^0(P_Z(\Lambda)\vert_Y),
    $$
    and there is a canonical surjective morphism
    $$
    p^Z_Y:P_Z(\Lambda)\longrightarrow P_Y(\Lambda).
    $$
    \item[(2)]For each $L^+G$-orbit $Gr_{\lambda}$, denote by $j_{\lambda}:Gr_{\lambda}\hookrightarrow Gr_G$  the inclusion map. The projective object $P_Z(\Lambda)$ has a filtration with the the associated graded
    $$
    gr(P_Z(\Lambda))\simeq \bigoplus_{Gr_{\lambda}\subset Z} \textnormal{CT}({^p j_{\lambda,*}}\underline{\Lambda}_{Gr_{\lambda}}[(2\rho,\lambda)])^*\otimes {^p j_{\lambda,!}}\underline{\Lambda}_{Gr_{\lambda}}[(2\rho,\lambda)].
    $$ In particular, $\textnormal{H}^{*}(P_Z(\Lambda))$ is a finitely generated free $\Lambda$-module.
    \item[(3)] For $\Lambda=\bar{\mathbb{Q}}_{\ell}$ and $\mathbb{F}_{\ell}$, there is a canonical isomorphism 
    $$
    P_Z(\Lambda)\simeq P_Z(\mathbb{Z}_{\ell})\otimes^L_{\mathbb{Z}_{\ell}} \Lambda.
    $$
\end{itemize}

\end{proposition}
Again, as the proof in \cite[Prp.10.1]{MV2} extends verbatim in our setting, we refer to \textit{loc.cit} for details. 
\\

For the rest of this section, we set $\Lambda=\mathbb{Z}_{\ell}$.

\begin{proposition}
Let $\mathcal{F}\in P_{L^+G}(Z,\Lambda)$ be a projective object. Then  $\textnormal{H}^{*}(\mathcal{F})$ is a projective $\Lambda$-module. In particular, $\textnormal{H}^{*}(\mathcal{F})$ is torsion-free.
\end{proposition}
\begin{proof}

Since $\textnormal{Hom}(P_Z(\Lambda),\bullet)$ is exact and faithful, the object $P_Z(\Lambda)$ is a projective generator of $P_{L^+G}(Z,\Lambda)$. Then each object in the Satake category admits a resolution by direct sums of $P_Z(\Lambda)$. Choose such a resolution for $\mathcal{F}$ 
\begin{equation}
P_Z(\Lambda)^{\oplus m}\longrightarrow \F\longrightarrow 0.
\end{equation}
In this way, $\F$ can be realized as a direct summand of $P_Z(\Lambda)^{\oplus m}$. By Proposition $5.3$ $(2)$, we notice that $\textnormal{H}^{*}(P_Z(\Lambda)^{\oplus m})$ is a finitely generated free $\Lambda$-module. Finally, by the exactness of the global cohomology functor $\textnormal{H}^{*}(\bullet)$, we conclude that $\textnormal{H}^{*}(\F)$ is a direct summand of $\textnormal{H}^{*}(P_Z(\Lambda)^{\oplus m})$ and is thus a projective $\Lambda$-module.
\end{proof}

\begin{remark}
The results established in Proposition $5.4$ become immediate once the geometric Satake equivalence is established.
\end{remark}
\section{The Monoidal Structure of $\textnormal{H}^*$}
In this section we study the $\mathbb{G}_m$-action on $\GG$ and apply the hyperbolic localization theorem to prove the hypercohomology functor $\textnormal{H}^*:P_{L^+G}(Gr_G,\Lambda)\rightarrow \textnormal{Mod}(\Lambda)$  is a monoidal functor. Then we study the relation between the global weight functor $\textnormal{CT}$ and the global cohomology functor $\textnormal{H}^*$. At the end of this section, we prove that the monoidal structure on $\textnormal{H}^{*}$ we constructed is compatible with the one constructed in \cite{Zh2}.
\\

Recall the action of $\mathbb{G}_m$ on $Gr_G$ defined in \S 4, and let $\mathbb{G}_m$ act on $Gr_G\times  Gr_G$ diagonally. Then, 
$$
    R\times R:=\{(g_1,g_2)\in Gr_G\times  Gr_G\vert L^+(2\rho^{\vee}(t))\cdot (g_1,g_2)=(g_1,g_2)\},
$$

$$
    S_{\mu_1}\times S_{\mu_2}=\{(g_1,g_2)\in Gr_G\times  Gr_G\vert \lim_{t\rightarrow 0} L^+(2\rho^{\vee}(t))\cdot(g_1,g_2)=(\varpi^{\mu_1},\varpi^{\mu_2})\},
$$
and
$$
    S_{\mu_1}^{-}\times S_{\mu_2}^{-}=\{(g_1,g_2)\in Gr_G\times  Gr_G\vert \lim_{t\rightarrow \infty} L^+(2\rho^{\vee}(t))\cdot(g_1,g_2)=(\varpi^{\mu_1},\varpi^{\mu_2})\}
$$
are the stable, attracting and repelling loci of the $\mathbb{G}_m$-action, respectively.
We write $(g_1\TT g_2)\in\GG$ for $(pr_1,m)^{-1}(g_1L^+G,g_1g_2L^+G)$ for any $(g_1 L^+G,g_1,g_2 L^+G)\in Gr_G\times Gr_G$. Define the action of $\mathbb{G}_m$ on $Gr_G\tilde{\times}Gr_G$ by $t (g_1\TT g_2):=(tg_1\TT g_1^{-1}g_2)$ for any $t\in \mathbb{G}_m$. Then the isomorphism  $(pr_1,m):Gr_G\tilde{\times}Gr_G\simeq Gr_G\times Gr_G$ is automatically $\mathbb{G}_m$-equivariant. The stable loci, attracting, and repelling loci of the $\mathbb{G}_m$-action on $\GG$ are 
$$
     R\TT R= \{(\varpi^{\mu_1}\TT\varpi^{\mu_2-\mu_1})\vert \mu_1,\mu_2\in \mathbb{X}_{\bullet}^+\},
$$

$$
    S_{\mu_1}\tilde{\times} S_{\mu_2-\mu_1}=\{(g_1\TT g_2)\in \GG \vert \lim_{t\rightarrow 0} L^+(2\rho^{\vee}(t))\cdot(g_1\TT g_2)=(\varpi^{\mu_1}\TT\varpi^{\mu_2-\mu_1})\},
$$
and
$$
    S_{\mu_1}^{-}\tilde{\times} S_{\mu_2-\mu_1}^{-}=\{((g_1\TT g_2)\in \GG \vert \lim_{t\rightarrow \infty} L^+(2\rho^{\vee}(t))\cdot(g_1\TT g_2)=(\varpi^{\mu_1}\TT\varpi^{\mu_2-\mu_1})\}
$$
respectively. 
\begin{lemma}
For any $\mathcal{F},\mathcal{G}\in \textnormal{Sat}_{G,\Lambda}$, we have the following isomorphisms
\begin{equation}
 \textnormal{H}^{*}_{S_{\mu_1}^{-}\tilde{\times} S_{\mu_2-\mu_1}^{-}}(\GG,\mathcal{F}\tilde{\boxtimes}\G)\simeq\textnormal{H}_c^{*}(\UU,\F\tilde{\boxtimes}\G)\simeq \textnormal{H}_c^{*}(S_{\mu_1},\F)\otimes \textnormal{H}_c^{*}(S_{\mu_2-\mu_1},\G).
\end{equation}
In addition, the above cohomology groups vanish outside degree $(2\rho,\mu_2)$
\end{lemma}

\begin{proof}
By our discussion on the $\mathbb{G}_m$-action on $Gr_G\tilde{\times}Gr_G$ above, the first isomorphism can be obtained by applying Braden's hyperbolic localization theorem as in \cite{DM}. Therefore we are left to prove the second isomorphism and the vanishing property of the cohomology. We first establish a canonical isomorphism 
$$
\textnormal{H}_c^{*}(\UU,\F\tilde{\boxtimes}\G)\cong \textnormal{H}_c^{*}(S_{\mu_1}\times S_{\mu_2-\mu_1},{^{p}}\textnormal{H}^{0}(\mathcal{F}\tilde{\boxtimes}\mathcal{G})).
$$
The idea of constructing this isomorphism is completely similar to the one that appears in \cite[Coro.2.17]{Zh2} and we sketch it here. 
\\

Assume $LU$ acts on $S_{\mu_1}$ via the quotient $L^nU$ for some positive integer $n$. Denote by $S_{\mu_1}^{(n)}$ the pushout of the $L^+U$-torsor $LU\rightarrow S_{\mu_1}$ along $L^+U\rightarrow L^nU$. Then $\pi:S^{(n)}_{\mu_1}\rightarrow S_{\mu_1}$ is an $L^nU$-torsor. Denote by $\pi^*\F$  the pullback of $\F$ along $\pi$. Then we have the following projection morphisms 
$$
S_{\mu_1}\times S_{\mu_2-\mu_1}\stackrel{\pi\times\textnormal{id}}{\longleftarrow} S_{\mu_1}^{(n)}\times S_{\mu_2-\mu_1}\stackrel{q}{\longrightarrow}S_{\mu_1}\tilde{\times}S_{\mu_2-\mu_1}.
$$
Since $L^nU$ is isomorphic to the perfection of an affine space of dimension $n\dim U$, we have the following canonical isomorprhisms

\begin{align*}
&  \textnormal{H}_c^{*}(\UU,\F\tilde{\boxtimes}\G)\\
 \cong & \textnormal{H}_{c}^{*}(S_{\mu_1}^{(n)}\times S_{\mu_2-\mu_1},q^*(\F\tilde{\boxtimes} \G)) \\
 \cong & \textnormal{H}_{c}^{*}(S_{\mu_1}^{(n)}\times S_{\mu_2-\mu_1},(\pi\times\textnormal{id})^*({^{p}}\textnormal{H}^{0}(\F\boxtimes \G))) \\
 \cong & \textnormal{H}_{c}^{*}(S_{\mu_1}\times S_{\mu_2-\mu_1}, {^{p}}\textnormal{H}^{0}(\F\boxtimes\G)). 
\end{align*}

Next, we prove there is a natural isomorphism 
\begin{equation}
\textnormal{H}_{c}^{*}(S_{\mu_1}\times S_{\mu_2-\mu_1}, {^{p}}\textnormal{H}^{0}(\F\boxtimes\G))\cong \textnormal{H}_c^{*}(S_{\mu_1},\F)\otimes\textnormal{H}_c^{*}(S_{\mu_2-\mu_1},\G).
\end{equation}
Assume $\G$ is a projective object in the Satake category. Then by Proposition $5.4$ and discussion in \S 2, we have ${^{p}}\textnormal{H}^{0}(\F\boxtimes\G)=\F\boxtimes\G$ and $(6.2)$ thus holds. Now we come back to the general situation. Note there is always a map from $\textnormal{H}_c^{*}(S_{\mu_1},\F)\otimes\textnormal{H}_c^{*}(S_{\mu_2-\mu_1},\G)$ to $\textnormal{H}_{c}^{*}(S_{\mu_1}\times S_{\mu_2-\mu_1}, {^{p}}\textnormal{H}^{0}(\F\boxtimes\G))$. In fact, let $a\in \textnormal{H}_c^{m}(S_{\mu_1},\F)$ and $b\in \textnormal{H}_c^{n}(S_{\mu_2-\mu_1},\G)$ be two arbitrary elements in the cohomology groups. Then $a$ and $b$ may be realized as 
$$
a:\underline{\Lambda}_{S_{\mu_1}}\rightarrow \F[m],\text{ and }b:\underline{\Lambda}_{S_{\mu_2-\mu_1}}\rightarrow \G[n].
$$
These two morphisms together induce a morphism
$$
a\boxtimes b:\underline{\Lambda}_{S_{\mu_1}\times S_{\mu_2-\mu_1}}\longrightarrow \F\boxtimes\G[m+n].
$$
Since $\F\boxtimes\G$ concentrates in non-positive perverse degrees, we can compose the above morphism with the natrual truncation morphism to get the following element
$$
a\boxtimes b:\underline{\Lambda}_{S_{\mu_1}\times S_{\mu_2-\mu_1}}\longrightarrow {^{p}}\textnormal{H}^{0}(\F\boxtimes G)[m+n]
$$
of $\textnormal{H}_{c}^{m+n}(S_{\mu_1}\times S_{\mu_2-\mu_1}, {^{p}}\textnormal{H}^{0}(\F\boxtimes\G))$. 
\\

By Corollary $5.2$, we can find a projective resolution $\cdots\rightarrow\mathcal{F}_2\rightarrow\mathcal{F}_1\rightarrow\mathcal{F}\rightarrow 0$ for $\mathcal{F}$. Since the functor $^p\textnormal{H}^0(\bullet\boxtimes\mathcal{G})$ is right exact, we get the following exact sequence
\begin{equation}
\cdots\longrightarrow {^{p}}\textnormal{H}^0(\mathcal{F}_2\boxtimes \mathcal{G})\longrightarrow {^{p}}\textnormal{H}^0(\mathcal{F}_1\boxtimes \mathcal{G})\longrightarrow {^{p}}\textnormal{H}^0(\mathcal{F}\boxtimes \mathcal{G})\longrightarrow 0.
\end{equation}
Recall the diagonal action of $\mathbb{G}_m$ on $Gr_G\times Gr_G$. We can apply the same argument as in Proposition $4.4$ to show that 
$$
\textnormal{H}^{*}(Gr_G\times Gr_G,\bullet)\simeq \oplus \textnormal{H}_c^{*}(S_{\mu_1}\times S_{\mu_2-\mu_1},\bullet)
$$
is an exact functor. As a result, the functor 
$$
\textnormal{H}_c^{*}(S_{\mu_1}\times S_{\mu_2-\mu_1},\bullet):P_{L^+G\times L^+G}(Gr_G\times Gr_G,\Lambda)\longrightarrow \textnormal{Mod}_{\Lambda}
$$
is also exact. 
Applying this functor to $(6.3)$ gives an exact sequence
$$
\cdots\rightarrow\textnormal{H}_c^*(S_{\mu_1},\F_2)\otimes \textnormal{H}_c^{*}( S_{\mu_2-\mu_1},\G)\rightarrow \textnormal{H}_c^*(S_{\mu_1},\F_1)\otimes \textnormal{H}_c^{*}(S_{\mu_2-\mu_1},\G)\rightarrow \textnormal{H}_c^*(S_{\mu_1}\times S_{\mu_2-\mu_1}, {^{p}}\textnormal{H}^{0}(\F\boxtimes\G))\rightarrow 0.
$$
Comparing the above exact sequence with the one obtained from tensoring the following exact sequence 
$$
\textnormal{H}_c^{*}(S_{\mu_1},\F_2)\longrightarrow\textnormal{H}_c^{*}(S_{\mu_1},\F_1)\longrightarrow\textnormal{H}_c^{*}(S_{\mu_1},\F)\longrightarrow 0
$$
with $\textnormal{H}_c^{*}(S_{\mu_2-\mu_1},\G)$, we complete the proof of $(6.2)$.
\\

Finally, consider Proposition $4.2$ together with $(6.2)$ and we conclude the proof of the lemma.
\end{proof}

The previous lemma motivates us to study the analogue of the total weight functor 
$$\textnormal{CT}':=\bigoplus_{\mu_1,\mu_2\in \mathbb{X}_{\bullet}}\textnormal{H}_c^*(\UU,\bullet):P_{L^+G}(Gr_G,\Lambda)\times P_{L^+G}(Gr_G,\Lambda)\longrightarrow \textnormal{Mod}(\mathbb{X}_{\bullet}).$$ 
Recall that we denote $F:\textnormal{Mod}(\mathbb{X}_{\bullet})\rightarrow \textnormal{Mod}_{\Lambda}$ to be the forgetful functor.

\begin{proposition}

There is a canonical isomorphism 
\begin{equation}
    \textnormal{H}^*(\GG,\F\tilde{\boxtimes}\G)\cong F\circ\textnormal{CT}'(\F\tilde{\boxtimes}\G):  P_{L^+G}(Gr_G,\Lambda)\times P_{L^+G}(Gr_G,\Lambda)\rightarrow \textnormal{Mod}_{\Lambda},
\end{equation}
for all $\F,G\in P_{L^{+}G}(Gr_G,\Lambda)$.
\end{proposition}

\begin{proof}
The convolution Grassmannian $\GG$ admits a stratification by the convolution of semi-infinite orbits 
$$
\{\UU \vert \mu_1,\mu_2\in \mathbb{X}_{\bullet}\}.
$$ 
For any $\F,\G\in P_{L^+G}(Gr_G,\Lambda)$,  there is a spectral sequence with $E_1$-terms $\textnormal{H}_c^{*}(S_{\mu_1}\TT S_{\mu_2-\mu_1},\F\tilde{\boxtimes}\G)$ and abutment $\textnormal{H}^{*}(\GG,\F\tilde{\boxtimes}\G)$. By the above lemma, it degenerates at the $E_1$ page. Hence, there exists a filtration

$$\textnormal{Fil}_{\geq \mu_1,\mu_2}\textnormal{H}^*(\mathcal{F}\tilde{\boxtimes}\G):=\textnormal{ker}(\textnormal{H}^*(\mathcal{F}\tilde{\boxtimes}\G)\rightarrow \textnormal{H}^*(S_{<\mu_1,<\mu_2},\mathcal{F}\tilde{\boxtimes}\G)),$$ where $S_{<\mu_1,<\mu_2}:=\displaystyle \mathop{\cup}_{\nu_1<\mu_1,\nu_1+\nu_2<\mu_2}S_{\nu_1}\tilde{\times} S_{\nu_2-\nu_1}$.
It is clear that the associated graded of this filtration is $\oplus_{\mu_1,\mu_2\in \mathbb{X}_{\bullet}}\textnormal{H}_c^*(\UU,\mathcal{F}\tilde{\boxtimes}\G)$. 
\\

Similarly, consider the stratification $\{S_{\mu_1}^{-}\tilde{\times} S_{\mu_2-\mu_1}^{-}\vert, \mu_1,\mu_2\in \mathbb{X}_{\bullet}\}$ of $Gr_G\tilde{\times}Gr_G$. It also induces a filtration

$$\textnormal{Fil}'_{<\mu_1,\mu_2} \textnormal{H}^*(\mathcal{F}\tilde{\boxtimes}\G):=\textnormal{Im}(\textnormal{H}^*_{T_{<\mu_1,<\mu_2}}(\mathcal{F}\tilde{\boxtimes}\G)\rightarrow \textnormal{H}^*(\mathcal{F}\tilde{\boxtimes}\G))$$
on $\textnormal{H}^*(Gr_G\tilde{\times}Gr_G,\mathcal{F}\tilde{\boxtimes}\mathcal{G})$ where $T_{<\mu_1,<\mu_2}:=\displaystyle \mathop{\cup}_{\nu_1<\mu_1,\nu_1+\nu_2<\mu_2}T_{\nu_1}\tilde{\times} T_{\nu_2-\nu_1}$
The two filtrations are complementary to each other by Lemma $6.1$ and the proposition is proved.
\end{proof}

\begin{proposition}
Under the canonical isomorphism 
$$
\textnormal{H}^{*}(Gr_G,\F\star\G)\cong \textnormal{H}^{*}(Gr_G\tilde{\times} Gr_G,\F\tilde{\boxtimes}\G),
$$
the weight functor decomposition of the hypercohomology functor obtained in Proposition $4.4$ and the analogous decomposition given by Proposition $6.2$ are compatible. More precisely, for any $\F,\G \in P_{L^+G}(Gr_G,\Lambda)$ any $\mu_2\in \mathbb{X}_{\bullet}$, we have the following isomorphism
\begin{equation}
   \textnormal{H}_c^*(S_{\mu_2},\F\star\G)\simeq \bigoplus_{\mu_1}\textnormal{H}_c^*(\UU, \F\tilde{\boxtimes}\G),
\end{equation}
  which identifies both sides as direct summands of the direct sum decomposition of $\textnormal{H}^{*}(Gr_G,\F\star\G)$ and $\textnormal{H}^{*}(Gr_G\tilde{\times} Gr_G,\F\tilde{\boxtimes}\G)$, respectively.
\end{proposition}
\begin{proof}
Consider the following commutative diagram
$$
\begin{tikzcd}[row sep=huge]
m^{-1}(S_{\mu_2}) \arrow[d, "m_1"]  \arrow[r,hook, "\tilde{f}^+"] & Gr_{G}\tilde{\times}Gr_{G} \arrow[d, "m"]
  \\
S_{\mu_2}   \arrow[r,hook, "f^+"]
 & Gr_{G} .
 \end{tikzcd}
 $$
Here, $f$ and $\tilde{f}^+$ are the natural locally closed embeddings. The morphism $m_1$ is the convolution morphism $m$ restricted to $m^{-1}(S_{\mu_2})$.
\\

Consider the $\mathbb{G}_m$-equivariant isomorphism $(pr_1,m):Gr_G\tilde{\times} Gr_G\simeq Gr_G\times Gr_G$. The preimage of $S_{\mu_2}$ along $m$ can be described as  
$$
(pr_1,m):m^{-1}(S_{\mu_2})\simeq Gr_G\times S_{\mu_2}.
$$
\\
As before, the diagonal action of $\mathbb{G}_m$ on $Gr_G\times S_{\mu_2}$  induces a $\mathbb{G}_m$- action on $m^{-1}(S_{\mu_2})$ with invariant loci $\{(\varpi^{\mu_1},\varpi^{\mu_2})\vert \mu_1\in \mathbb{X}_{\bullet}\}$. Via the isomorphism $(pr_1,m)^{-1}$, the attracting and repelling loci for $(\varpi^{\mu_1}\TT\varpi^{\mu_2-\mu_1})$  in $m^{-1}(S_{\mu_2})$ are 
$$
S_{\mu_1}\tilde{\times} S_{\mu_2-\mu_1},
$$ 
and 
$$
T_{\mu_1,\mu_2}:=(pr_1,m)^{-1}(S_{\mu_1}^{-}\times \{\varpi^{\mu_2}\})
$$
respectively.
Applying the hyperbolic localization theorem to $m^{-1}(S_{\mu_2})$, we have the following isomorphism
\begin{equation}
\textnormal{H}_c^{*}(S_{\mu_1}\tilde{\times} S_{\mu_2-\mu_1},\mathcal{F}\tilde{\boxtimes}\mathcal{G})\simeq \textnormal{H}_{T_{\mu_1,\mu_2}}^{*}(\mathcal{F}\tilde{\boxtimes}\mathcal{G}).
\end{equation}
By Lemma $6.1$, the above cohomology groups concentrate in a single degree. 
\\

Filtering the space $m^{-1}(S_{\mu_2})$ by $\{S_{\mu_1}\tilde{\times}S_{\mu_2-\mu_1}\vert \mu_1\in\mathbb{X}_{\bullet}\}$, we get a spectral sequence with $E_1$-terms $\textnormal{H}_c^{*}(S_{\mu_1}\tilde{\times}S_{\mu_2-\mu_1},\F\tilde{\boxtimes}\G)$. As noticed in Lemma $6.1$, this spectral sequence degenerates at $E_1$-page. Then, there exists a filtration
$$
\textnormal{Fil}_{\mu_1,\mu_2'}:=\textnormal{Ker}(\textnormal{H}^{*}(m^{-1}(S_{\mu_2}),\F\tilde{\boxtimes}\G)\rightarrow \textnormal{H}^{*}(\cup_{\mu_1'<\mu_1} S_{\mu_1'}\tilde{\times}S_{\mu_2-\mu_1'},\F\tilde{\boxtimes}\G))
$$ 
with associated graded
$$
\bigoplus_{\mu_1}\textnormal{H}_c^{*}(\UU,\F\tilde{\boxtimes}\G).
$$
Similarly, filtering $m^{-1}(S_{\mu_2})$ by $\{T_{\mu_1,\mu_2}\vert \mu_1\in\mathbb{X}_{\bullet}\}$, we get an induced spectral sequence with $E_1$-terms $\textnormal{H}^{*}_{T_{\mu_1,\mu_2}}(\F\tilde{\boxtimes}\G)$. This spectral sequence also degenerates at the $E_1$-page and there is an induced filtration
$$
\textnormal{Fil}'_{\mu_1,\mu_2}:=\textnormal{Im}(\textnormal{H}_{T_{<\mu_1,\mu_2}}^{*}(\F\tilde{\boxtimes}\G)\rightarrow\textnormal{H}^{*}(\F\tilde{\boxtimes}\G)),
$$
where $T_{<\mu_1,\mu_2}:=\cup_{\mu_1'<\mu_1}T_{\mu_1',\mu_2}$. The two filtrations are complementary to each other by $(6.6)$ and together defines the decomposition $\textnormal{H}_c^*(S_{\mu_2},\F\star\G)\simeq \bigoplus_{\mu_1}\textnormal{H}_c^*(\UU, \F\tilde{\boxtimes}\G)$.

\end{proof}
By the above proposition, Proposition $6.2$ induces a monoidal structure of the functor $\textnormal{H}^{*}$.
\begin{proposition}
The hypercohomology functor $\textnormal{H}^*(Gr_G,\bullet):P_{L^+G}(Gr_G,\Lambda)\rightarrow \textnormal{Mod}_{\Lambda}$ is a monoidal functor. In addition, the obtained monoidal structure is compatible with the weight functor decomposition established in Proposition $4.2$ 
\end{proposition}

\begin{proof}
Recall for $\mathcal{F},\mathcal{G}\in P_{L^+G}(Gr_G,\Lambda)$, the convolution product $\F\star \G$ is defined as $\F\star \G=Rm_!(\F\tilde{\boxtimes}\G)$. Then by Lemma $6.1$ and Proposition $6.2$, there are canonical isomorphisms
\begin{align*}
   & \textnormal{H}^{*}(Gr_G,\F\star\G)\\
     \cong & \textnormal{H}^{*}(\GG,\F\tilde{\boxtimes}\G)\\
    \cong & \bigoplus_{\mu_1,\mu_2}\textnormal{H}_c^{*}(\UU,\F\tilde{\boxtimes}\G)\\
    \cong & \bigoplus_{\mu_1,\mu_2}\Big (\textnormal{H}_c^{*}(S_{\mu_1},\F)\otimes\textnormal{H}_c^{*}(S_{\mu_2-\mu_1},\G)\Big )\\
    \cong & \Big( \bigoplus_{\mu_1}\textnormal{H}_c^{*}(S_{\mu_1},\F)\Big )\otimes \Big (\bigoplus_{\mu_2}\textnormal{H}_c^{*}(S_{\mu_2},\G)\Big )\\
    \cong & \textnormal{H}^{*}(\F)\otimes \textnormal{H}^{*}(\G).
\end{align*}
Note by Proposition $4.4$, we have the decomposition of the total weight functor into direct sum of weight functors $\textnormal{H}^*(Gr_G,\F\star\G)\simeq \oplus_{\lambda}\textnormal{H}_c^*(S_{\lambda},\F\star\G)$. Proposition $6.3$ then shows the monoidal structure obtained above is compatible with the weight functor decomposition.
\\

Finally, we need to show that the monoidal structure of $\textnormal{H}^{*}$ is compatible with the associativity constraint. This can be proved by considering the $\mathbb{G}_m$- action on $Gr_G\tilde{\times}Gr_G\tilde{\times} Gr_G$ induced by the diagonal action of $\mathbb{G}_m$ on $Gr_G\times Gr_G\times Gr_G$ via the isomorphism 
$$
(m_1,m_2,m_3)^{-1}: Gr_G\times Gr_G\times Gr_G\simeq Gr_G\tilde{\times}Gr_G\tilde{\times} Gr_G.
$$ 
Note in this case we can still split the intersection $(S_{\nu_1}\tilde{\times}S_{\nu_2}\tilde{\times}S_{\nu_3})\cap (Gr_{\leq\mu_1}\tilde{\times}Gr_{\leq\mu_2}\tilde{\times}Gr_{\leq\mu_3})$ by $(4.1)$. This allows us to apply the hyperbolic localization theorem and a similar spectral sequence argument as before. We obtain the desired compatibility property and the proposition is thus proved.

\end{proof}

With the monoidal structure of $\textnormal{H}^{*}$ established above, we are now ready to prove the following results.
\begin{proposition}
For any $\mathcal{F}\in \textnormal{Sat}_{G,\Lambda}$, the functors $(\bullet)\star\mathcal{F}$ and $\mathcal{F}\star(\bullet)$ are both right exact. If in addition $\mathcal{F}$ is a projective object, the these functors are exact.
\end{proposition}
\begin{proof}
Let 
\begin{equation}
0\rightarrow\mathcal{G}'\rightarrow\mathcal{G}\rightarrow\mathcal{G}''\rightarrow 0
\end{equation}
be an exact sequence in $\textnormal{Sat}_{G,\Lambda}$. By Proposition $4.4$, taking global cohomology gives an exact sequence
\begin{equation}
    \textnormal{H}^{*}(\mathcal{G}')\longrightarrow\textnormal{H}^{*}(\mathcal{G})\longrightarrow\textnormal{H}^{*}(\mathcal{G}'')\longrightarrow 0.
\end{equation}
Tensoring $(6.8)$ with $\textnormal{H}^{*}(\F)$ gives exact sequence 
\begin{equation}
    \textnormal{H}^{*}(\mathcal{G}')\otimes\textnormal{H}^{*}(\mathcal{F})\longrightarrow\textnormal{H}^{*}(\mathcal{G})\otimes\textnormal{H}^{*}(\mathcal{F})\longrightarrow\textnormal{H}^{*}(\mathcal{G}'')\otimes\textnormal{H}^{*}(\mathcal{F})\longrightarrow 0.
\end{equation}
By proposition $6.3$, $(6.9)$ is canonically isomorphic to the following sequence
\begin{equation}
    \textnormal{H}^{*}(\G'\star\F)\longrightarrow\textnormal{H}^{*}(\G\star\F)\longrightarrow\textnormal{H}^{*}(\G''\star\F)\longrightarrow 0.
\end{equation}
Notice that by Proposition $4.4$, the global cohomology functor $\textnormal{H}^{*}(\bullet)$ is faithful, then the exactness of $(6.9)$ implies the sequence 
$$
\mathcal{G}'\star\mathcal{F}\longrightarrow\mathcal{G}\star\mathcal{F}\longrightarrow\mathcal{G}''\star\mathcal{F}\longrightarrow 0
$$ is also exact. The right exactness for $\mathcal{F}\star(\bullet)$
can be proved similarly.
\\

Now, assume $\F$ to be a projective object in the Satake category. By Proposition $5.4$, we know that the functors $(\bullet)\otimes \textnormal{H}^{*}(\mathcal{F})$ and $\textnormal{H}^{*}(\mathcal{F})\otimes (\bullet)$ are both exact. Then argue as before and use the monoidal structure and the faithfulness of the functor $\textnormal{H}^{*}(\bullet)$, we conclude the proof.
\end{proof}

We conclude the discussion on the monoidal structure of $\textnormal{H}^{*}$ by identifying it with the one constructed in \cite{Zh2}. For this purpose, we briefly recall the construction in \textit{loc.cit}.
\\

Let $\mathcal{F},\mathcal{G}\in P_{L^+G}(Gr_G,\bar{\mathbb{Q}}_{\ell})$. Assume that $L^+G$ acts on $\textnormal{supp}(\mathcal{G})$ via the quotient $L^+G\rightarrow L^mG$. Define $\textnormal{supp}(\mathcal{F})\tilde{\times}\textnormal{supp}(\mathcal{G}):=\textnormal{supp}(\mathcal{F})^{(m)}\times ^{L^mG}\textnormal{supp}(\mathcal{G})$ and denote by $\pi$ the projection morphism $\textnormal{supp}(\mathcal{F})^{(m)}\rightarrow \textnormal{supp}(\mathcal{F})$. Then we have an $L^+G\times L^mG$-equivariant projection morphism 
$$
p:\textnormal{supp}(\mathcal{F})^{(m)}\times \textnormal{supp}(\mathcal{G})\longrightarrow \textnormal{supp}(\mathcal{F})\tilde{\times}\textnormal{supp}(\mathcal{G})
$$
where $L^+G$ acts on $\textnormal{supp}(\mathcal{F})^{(m)}$ by multiplication on the left and $L^mG$ acts on $\textnormal{supp}(\mathcal{F})^{(m)}\times \textnormal{supp}(\mathcal{G})$ diagonally from the middle. Then $p$ induces a canonical isomorphism of the $L^+G$-equivariant cohomology (cf.\cite{Zh2} A.3.5)
\begin{equation}
\textnormal{H}_{L^+G}^{*}(\textnormal{supp}(\mathcal{F})\tilde{\times}\textnormal{supp}(\mathcal{G}),\mathcal{F}\tilde{\boxtimes}\mathcal{G})\cong \textnormal{H}_{L^+G\times L^mG}^{*}(\textnormal{supp}(\mathcal{F})^{(m)}\times\textnormal{supp}(\mathcal{G}),\pi^{*}\mathcal{F}\boxtimes\mathcal{G}).
\end{equation}
By the equivariant K\"{u}nneth formula (cf. \cite[A.1.15]{Zh1} ), there is a canonical isomorphism
\begin{equation}
\textnormal{H}_{L^+G\times L^mG}^{*}(\textnormal{supp}(\mathcal{F})^{(m)}\times\textnormal{supp}(\mathcal{G}),\pi^{*}\mathcal{F}\boxtimes\mathcal{G})\cong \textnormal{H}_{L^+G\times L^mG}^{*}(\textnormal{supp}(\mathcal{F})^{(m)},\mathcal{F})\otimes \textnormal{H}_{L^+G\times L^mG}^{*}(\textnormal{supp}(\mathcal{G}),\mathcal{G}).
\end{equation}
Combine $(6.11)$ with $(6.12)$ and we conclude a canonical isomorphism 
\begin{equation}
    \textnormal{H}_{L^+G}^{*}(\textnormal{supp}(\mathcal{F})\tilde{\times}\textnormal{supp}(\mathcal{G}),\mathcal{F}\tilde{\boxtimes}\mathcal{G})\cong \textnormal{H}^{*}_{L^+G}(\textnormal{supp}(\mathcal{F}),\mathcal{F})\otimes \textnormal{H}^{*}_{L^+G}(\textnormal{supp}(\mathcal{G}),\mathcal{G}).
\end{equation}
We denote by $\bar{G}_{\bar{\mathbb{Q}}_{\ell}}$ the base change of $\bar{G}$ to $\bar{\mathbb{Q}}_{\ell}$.
Let $R_{\bar{G},\ell}:=\textnormal{Sym}(\mathfrak{g}_{\bar{\mathbb{Q}}_{\ell}}(-1))^{G_{\bar{\mathbb{Q}}_{\ell}}}$ denote the algebra of invariant polynomials on the
Lie algebra $\mathfrak{g}_{\bar{\mathbb{Q}}_{\ell}}(-1)$. Then $(6.13)$ induces an isomorphism of $R_{\bar{G},\ell}$-bimodules. In addition, the two $R_{\bar{G},\ell}$-module structures coincide (\cite{Zh2} Lemma $2.19$) and the base change of $(6.13)$ along the argumentation map $R_{\bar{G},\ell}\rightarrow \bar{\mathbb{Q}}_{\ell}$, the canonical isomorphism 
\begin{equation}
\textnormal{H}^{*}_{L^+G}(\mathcal{F})\otimes_{R_{\bar{G},\ell}}\bar{\mathbb{Q}}_{\ell}\simeq \textnormal{H}^{*}(\mathcal{F}).
\end{equation} 
gives the monoidal structure of $\textnormal{H}^{*}$ in the $\bar{\mathbb{Q}}_{\ell}$-case (\cite{Zh2}, Proposition $2.20$).
\\

Then to identify the monoidal structures, it suffices to prove the following proposition.
\begin{proposition}
Let $\mathcal{F},\mathcal{G}\in P_{L^+G}(Gr_G,\mathbb{Z}_{\ell})$ be two projective objects. We denote $\mathcal{F}\otimes \bar{\mathbb{Q}}_{\ell}$ and $\mathcal{G}\otimes \bar{\mathbb{Q}}_{\ell}$ by $\F'$ and $\G'$, respectively. Then the following diagram commutes
\begin{equation}
\begin{tikzcd}[row sep=huge]
\textnormal{H}_{L^+G}^{*}(\textnormal{supp}(\mathcal{F}')\tilde{\times}\textnormal{supp}(\mathcal{G}'),\mathcal{F}'\tilde{\boxtimes}\mathcal{G}') \otimes_{R_{\bar{G},\ell}} \bar{\mathbb{Q}}_{\ell} \arrow[r,"(6.14)"] \arrow[d,"(6.13)"] & \textnormal{H}^{*}(\mathcal{F}'\tilde{\boxtimes} \mathcal{G}')\arrow[d,"\alpha"] \\
(\textnormal{H}_{L^+G}^{*}(\mathcal{F}')\otimes_{R_{\bar{G},\ell}} \textnormal{H}_{L^+G}^{*}(\mathcal{G}')) \otimes_{R_{\bar{G},\ell}} \bar{\mathbb{Q}}_{\ell} \arrow[d,"\cong"] & \bigoplus_{\mu_1,\mu_2}\textnormal{H}_c^{*}(S_{\mu_1}\tilde{\times}S_{\mu_2-\mu_1},\mathcal{F}'\tilde{\boxtimes}\mathcal{G}') \arrow[d,"\beta"] \\
(\textnormal{H}_{L^+G}^{*}(\mathcal{F}')\otimes_{R_{\bar{G},\ell}}\bar{\mathbb{Q}}_{\ell})\otimes_{\bar{\mathbb{Q}}_{\ell}}(\textnormal{H}_{L^+G}^{*}(\mathcal{G}') \otimes_{R_{\bar{G},\ell}} \bar{\mathbb{Q}}_{\ell}) \arrow[r,"\simeq"] & (\bigoplus_{\mu_1}\textnormal{H}_c^{*}(S_{\mu_1},\mathcal{F}'))\otimes _{\bar{\mathbb{Q}}_{\ell}} (\bigoplus_{\mu_2}\textnormal{H}_c^{*}(S_{\mu_2-\mu_1},\mathcal{G}'))
\end{tikzcd}
\end{equation}
where the morphisms $\alpha$ and $\beta$ are the base change of isomorphisms $(6.4)$ and $(6.1)$ to $\bar{\mathbb{Q}}_{\ell}$, respectively.
\end{proposition}

\begin{proof}
Consider the filtrations 
$$
\textnormal{Fil}_{\geq\mu_1,\mu_2}\textnormal{H}^{*}(\mathcal{F}'\tilde{\boxtimes}\mathcal{G}'), \textnormal{Fil}_{\geq\mu}\textnormal{H}^{*}(\mathcal{F}'),\text{ and } \textnormal{Fil}_{\geq\mu}\textnormal{H}^{*}(\mathcal{G}')
$$ defined as in Proposition $6.2$ and Proposition $4.4$. To prove the proposition, it suffices to prove these filtrations respect $(6.13)$. Then taking the Verdier duality (note now $\textnormal{H}^{*}(\mathcal{F}'\tilde{\boxtimes}\mathcal{G}')$, $\textnormal{H}^{*}(\mathcal{F}')$, and $\textnormal{H}^{*}(\mathcal{G}')$ are all $\bar{\mathbb{Q}}_{\ell}$-vector spaces) implies that the complementary filtrations $\textnormal{Fil}'_{<\mu_1,\mu_2}$ and $\textnormal{Fil}'_{<\mu}$ also respect $(6.13)$. This will provide the commutativity of $(6.15)$.
\\

The approach we will use is similar to the one given in \cite[Proposition 5.3.14]{Zh1} and we sketch it here. Although the semi-infinite orbit $S_{\mu}$ does not admit an $L^+G$-action, it is stable under the action of the constant torus $T\subset L^+T\subset L^+G$. Then so does the convolution product of semi-infinite orbits $S_{\mu_1}\tilde{\times} S_{\mu_2-\mu_1}$.  Stratifying $Gr_G\tilde{\times}Gr_G$ by $\{S_{\mu_1}\tilde{\times} S_{\mu_2-\mu_1}\vert \mu_1,\mu_2\in \mathbb{X}_{\bullet}\}$, we get a spectral sequence with $E_1$-terms $\textnormal{H}_{T,c}^{*}(S_{\mu_1}\tilde{\times} S_{\mu_2-\mu_1},\mathcal{F}'\tilde{\boxtimes}\mathcal{G}')$ which abuts to $\textnormal{H}_T^{*}(\mathcal{F}'\tilde{\boxtimes}\mathcal{G}')$. By \cite[Proposition 2.7]{Zh2}, the spectral sequence degenerates at the $E_1$-page and the filtration $\textnormal{Fil}_{\geq\mu_1,\mu_2}$ thus lifts to a new filtration of $\textnormal{H}_{T}^{*}$
$$
\textnormal{Fil}_{\geq\mu_1,\mu_2}\textnormal{H}_{T}^{*}(\mathcal{F}'\tilde{\boxtimes}\mathcal{G}'):=\textnormal{ker}(\textnormal{H}^*_T(\mathcal{F}'\tilde{\boxtimes}\G')\rightarrow \textnormal{H}^*_T(S_{<\mu_1}\tilde{\times} S_{<\mu_2-\mu_1},\mathcal{F}'\tilde{\boxtimes}\G')).
$$
Use a similar argument as in the proof of Proposition $6.2$, the associated graded of this filtration equals $\bigoplus_{\mu_1,\mu_2} \textnormal{H}_{T,c}^{*}(S_{\mu_1}\tilde{\times} S_{\mu_2-\mu_1},\mathcal{F}'\tilde{\boxtimes}\mathcal{G}')$. Note all the terms in this filtration and the associated graded are in fact free $R_{\bar{T},\ell}$-modules, then base change to $\bar{\mathbb{Q}}_{\ell}$ along the argumentation map $R_{\bar{T},\ell}\rightarrow \bar{\mathbb{Q}}_{\ell}$ recovers our original filtration $\textnormal{Fil}_{\geq\mu_1,\mu_2}$. Similarly, we can define the filtrations $\textnormal{Fil}_{\geq\mu}\textnormal{H}_T^{*}(\mathcal{F}')$ and $\textnormal{Fil}_{\geq\mu}\textnormal{H}_T^{*}(\mathcal{G}')$ which recover the original filtrations $\textnormal{Fil}_{\geq\mu}\textnormal{H}^{*}(\mathcal{F}')$ and $\textnormal{Fil}_{\geq\mu}\textnormal{H}^{*}(\mathcal{G}')$ in the same way.
\\

Since 
$$
\textnormal{H}_T^{*}(\bullet)\simeq \textnormal{H}_{L^+G}^{*}(\bullet)\otimes_{R_{\bar{G},\ell}} R_{\bar{T},\ell}:P_{L^+G}(Gr_G,\bar{\mathbb{Q}}_{ell})\longrightarrow \textnormal{Vect}_{\bar{\mathbb{Q}}_{\ell}},
$$
then $(6.13)$ induces a monoidal structure on the $T$-equivariant cohomology 
\begin{equation}
\textnormal{H}_T^{*}(\mathcal{F}'\star{G}')\simeq \textnormal{H}_T^{*}(\mathcal{F}')\otimes \textnormal{H}_T^{*}(\mathcal{G}').
\end{equation}
Then we are left to show that $(6.16)$ is compatible with the filtrations $\textnormal{Fil}_{\geq\mu_1,\mu_2}$ and $\textnormal{Fil}_{\geq\mu}$. It suffices to check the compatibility with filtrations $\textnormal{Fil}_{\geq\mu_1,\mu_2}\textnormal{H}_T^{*}$ and $\textnormal{Fil}_{\geq\mu}\textnormal{H}_T^{*}$ over the generic point of $\textnormal{Spec}R_{\bar{T},\ell}$. Denote 
$$
\textnormal{H}_{\lambda}:=\textnormal{H}_{T}^{*}\otimes_{R_{\bar{T},\ell}}Q
$$
where $Q$ is the fraction filed of $R_{\bar{T},\ell}$.
By the equivariant localization theorem, we have isomorphisms
$$
\textnormal{H}_{\lambda}(\mathcal{F}'\tilde{\boxtimes}\mathcal{G}')\simeq \bigoplus_{\mu_1,\mu_2}\textnormal{H}_{\lambda}(\mathcal{F}'\tilde{\boxtimes}\mathcal{G}'\vert_{(\varpi^{\mu_1}\TT\varpi^{\mu_2-\mu_1})}),
$$
and
$$
\textnormal{H}_{\lambda}^{*}(S_{<\mu_1,<\mu_2}\,\mathcal{F}'\tilde{\boxtimes}\mathcal{G}')\simeq  \bigoplus_{\nu_1<\mu_1,\nu_2<\mu_2}\textnormal{H}_{\lambda}(\mathcal{F}'\tilde{\boxtimes}\mathcal{G}'\vert_{(\varpi^{\nu_1}\TT\varpi^{\nu_2-\nu_1})}).
$$
Then it follows that 
$$
\textnormal{Fil}_{\geq\mu_1,\mu_2}\textnormal{H}_{\lambda}(\mathcal{F}'\tilde{\boxtimes}\mathcal{G}'):=\textnormal{Fil}_{\geq\mu_1,\mu_2}\textnormal{H}_{T}^{*}(\mathcal{F}'\tilde{\boxtimes}\mathcal{G}')\otimes_{R_{\bar{T},\ell}} Q \simeq \bigoplus_{\nu_1\geq\mu_1,\nu_2\geq\mu_2}\textnormal{H}_{\lambda}(\mathcal{F}'\tilde{\boxtimes}\mathcal{G}'\vert_{(\varpi^{\nu_1}\TT\varpi^{\nu_2-\nu_1})}).
$$ 
Applying the equivariant localization theorem again gives isomorphisms
$$
\textnormal{H}_{\lambda}(\mathcal{F}')\simeq\bigoplus _{\mu}\textnormal{H}_{\lambda}(\mathcal{F}'\vert_{\varpi^{\mu}}), 
$$
and
$$
\textnormal{H}_{\lambda}(S_{<\mu},\mathcal{F}')\simeq\bigoplus _{\nu<\mu}\textnormal{H}_{\lambda}(\mathcal{F}'\vert_{\varpi^{\nu}}).
$$
Similarly we get a filtration  $\textnormal{Fil}_{\geq\mu}\textnormal{H}_{\lambda}(\mathcal{F}')\simeq \oplus_{\nu\geq\mu} \textnormal{H}_{\lambda}(\mathcal{F}'\vert_{\varpi^{\nu}})$ induced by $\textnormal{Fil}_{\geq\mu}\textnormal{H}_T^{*}(\F)$. 
\\

Notice that as for $\textnormal{H}_{L^+G}^{*}(\bullet)$, the monoidal structure $(6.16)$ is defined via the composition of the following isomorphisms
\begin{align*}
    & \textnormal{Fil}_{\geq\mu_1,\mu_2}\textnormal{H}_{\lambda}(\mathcal{F}'\tilde{\boxtimes}\mathcal{G}')\\
 \simeq & \bigoplus_{\nu_1\geq\mu_1,\nu_2\geq\mu_2}\textnormal{H}_{\lambda}(\mathcal{F}'\tilde{\boxtimes}\mathcal{G}'\vert_{(\varpi^{\nu_1}\TT\varpi^{\nu_2-\nu_1})})\\
 \simeq & \bigoplus_{\nu_1\geq\mu_1,\nu_2\geq\mu_2}\textnormal{H}_{\lambda}(\mathcal{F}'\vert_{\varpi^{\nu_1}})\otimes \textnormal{H}_{\lambda}(\mathcal{G}'\vert_{\varpi^{\nu_2-\nu_1}})\\
 \simeq & \bigoplus_{\nu_1\geq\mu_1}\textnormal{H}_{\lambda}(\mathcal{F}'\vert_{\varpi^{\nu_1}})\otimes \bigoplus_{\nu_2\geq\mu_2-\mu_1}\textnormal{H}_{\lambda}(\mathcal{G}'\vert_{\varpi^{\nu_2}}),\\
 \simeq & \textnormal{Fil}_{\geq\mu_1}\textnormal{H}_{\lambda}(\mathcal{F}')\bigotimes \textnormal{Fil}_{\geq\mu_2-\mu_1}\textnormal{H}_{\lambda}(\mathcal{G}')
\end{align*}
where the second isomorphism is obtained by an analogue of $(6.13)$ for $T$-equivariant cohomology and the equivariant K\"{u}nneth formula. Note the monoidal structure of the total weight functor $\textnormal{CT}$ is compatible with that of the hypercohomology cohomology functor $\textnormal{H}^{*}$ by Proposition $6.3$,  we thus conclude the proof.
\end{proof}

\section{Tannakian Construction}
Let $Z\subset Gr_G$ denote a closed subspace consisting of a finite union of $L^+G$-orbits. Then any $\F\in P_{L^+G}(Z,\Lambda)$ admits a presentation
$$
P_1\longrightarrow P_0\longrightarrow \F\longrightarrow 0,$$
where $P_1$ and $P_0$ are finite direct sums of $P_Z(\Lambda)$. Write $A_Z(\Lambda)$ for $\textnormal{End}_{P_{L^+G}(Z,\Lambda)}(P_Z(\Lambda))^{op}$. Since  by Proposition $5.3.(2)$ $A_Z(\Lambda)$ is a finite free $\Lambda$-module, any finitely generated $A_Z(\Lambda)$-module is also finitely presented. Now we recall the following version of Gabriel and Mitchell's theorem as formulated in \cite[Theorem 9.1]{BR}.

\begin{thm}
Let $\mathcal{C}$ be an abelian category. Let $P$ be a projective object and write $A=\textnormal{End}_{\mathcal{C}}(P)^{op}$. Denote $\mathcal{M}$ to be the full subcategory of $\mathcal{C}$ consisting of objects $M$ which admits a presentation 
$$
P_1\longrightarrow P_0\longrightarrow M\longrightarrow 0,$$
where $P_1$ and $P_0$ are finite direct sums of $P$. Let $\mathcal{M}'_A$ be the category of finitely presented right A-modules. Then
\begin{itemize}
    \item [(1)] there is an equivalence of abelian categories $\mathcal{M}\simeq \mathcal{M}'$ induced by the functor $\textnormal{Hom}_{\mathcal{C}}(P,\bullet)$
    \item[(2)] there is a canonical isomorphism between the endomorphism ring of the functor $\textnormal{Hom}_{\mathcal{C}}(P,\bullet)$ and $A^{op}$.
\end{itemize}
\end{thm}
The above theorem and the discussion before it enable us to deduce an equivalence of abelian categories
$$
E_Z:P_{L^+G}(Z,\Lambda)\simeq \mathcal{M}'_{A_Z(\Lambda)}.
$$
Let $i:Y\hookrightarrow Z$ be an inclusion of closed subsets consisting of $L^+G$-orbits, then we have the functor $i_*:P_{L^+G}(Y,\Lambda)\rightarrow P_{L^+G}(Z,\Lambda)$. In addition,  $i_*$ induces a functor $(i_Y^Z)^*:\mathcal{M}'_{A_Y(\Lambda)}\rightarrow \mathcal{M}'_{A_Z(\Lambda)}$ which in turn gives a ring homomorphism $i^Z_Y:A_Z(\Lambda)\rightarrow A_Y(\Lambda)$. Note for any $a\in A_Z(\Lambda)$ and $\F\in P_{L^+G}(Y,\Lambda)$, we have canonical isomorphisms
\begin{align*}
 & a\cdot  E_Z(i_*\F)\\
 \simeq & a \cdot \textnormal{Hom}(P_Z(\Lambda),i_*\F))\\
  \simeq & a \cdot (i^Z_Y)^* (\textnormal{Hom}(P_Z(\Lambda),i_*\F))\\
 \simeq & i_Y^Z(a) \cdot \textnormal{Hom}(P_Y(\Lambda),\F))\\
 \simeq & i_Y^Z(a)\cdot E_Y(\F).
 \end{align*}

Define $B_Z(\Lambda):=\textnormal{Hom} (A_Z(\Lambda),\Lambda))$. Since $A_Z(\Lambda)$ is a finite free $\Lambda$-module, then so is $B_Z(\Lambda) $ and we have the following canonical equivalence of abelian categories
$$
\mathcal{M}'_{A_Z(\Lambda)}\simeq \textnormal{Comod}_{B_Z(\Lambda)}.
$$
The dual map of $i^Z_Y$ gives a map $\iota_Z^Y:B_Y(\Lambda)\rightarrow B_Z(\Lambda)$. Let $B(\Lambda)=\varinjlim B_{Z}(\Lambda)$, we conclude that $\textnormal{Sat}_{G,\Lambda}\simeq \textnormal{Comod}_{B(\Lambda)}$ as abelian categories. Moreover, by Proposition $5.3.(3)$ we know that 
\begin{equation}
    B(\Lambda)\simeq B(\mathbb{Z}_{\ell})\otimes_{\mathbb{Z}_{\ell}}\Lambda
\end{equation}
for $\Lambda=\bar{\mathbb{Q}}_{\ell}$ and $\mathbb{F}_{\ell}$.
\\

Take any $\mu\in \mathbb{X}_{\bullet}^{+}$, write $A_{\mu}(\Lambda)$ and $B_{\mu}(\Lambda)$ for $A_{Gr_{\leq\mu}}(\Lambda)$ and $B_{Gr_{\leq\mu}}(\Lambda)$, respectively. For any $\mu,\nu\in \mathbb{X}_{\bullet}^{+}$ such that $\mu\leq \nu$, use notations $i^{\nu}_{\mu}$ and $\iota_{\nu}^{\mu}$ for $i^{Gr_{\leq\nu}}_{Gr_{\leq\mu}}$ and $\iota^{Gr_{\leq\mu}}_{Gr_{\leq\nu}}$, respectively. Also, we denote by $P_{\theta}(\Lambda)$ the projective object $P_{Gr_{\leq\theta}(\Lambda)}$ for any $\theta\in\mathbb{X}_{\bullet}^{+}$. Note the following canonoical isomorphism by the monoidal structure of $\textnormal{H}^{*}$ established by Proposition $6.4$
\begin{align*}
   & \textnormal{Hom}(P_{\mu+\nu}(\Lambda),P_{\mu}\star P_{\nu})\\
   \simeq & \textnormal{H}^{*}(P_{\mu}\star P_{\nu}) \\
   \simeq & \textnormal{H}^{*}(P_{\mu})\otimes \textnormal{H}^{*}(P_{\nu}) \\
   \simeq & \textnormal{Hom}(P_{\mu}(\Lambda),P_{\mu}(\Lambda))\otimes\textnormal{Hom}(P_{\nu}(\Lambda),P_{\nu}(\Lambda))
   \end{align*}
Then the element $id_{P_{\mu}(\Lambda)}\otimes id_{P_{\nu}(\Lambda)}\in \textnormal{Hom}(P_{\mu}(\Lambda),P_{\mu}(\Lambda))\otimes\textnormal{Hom}(P_{\nu}(\Lambda),P_{\nu}(\Lambda))$ gives rise to a morphism 
$$
f_{\mu,\nu}:P_{\mu+\nu}(\Lambda)\longrightarrow P_{\mu}(\Lambda)\star P_{\nu}(\Lambda).
$$
Applying the functor $\textnormal{H}^{*}$ and dualizing, we get a morphism
$$
g_{\mu,\nu}: B_{\mu}(\Lambda)\otimes B_{\nu}(\Lambda)\rightarrow B_{\mu+\nu}(\Lambda).
$$
We check that the multiplication maps $g_{\bullet,\bullet}$ is compatible with the maps $\iota^{\bullet}_{\bullet}$ i.e. for any $\mu\leq\mu',\nu\leq\nu'\in \mathbb{X}_{\bullet}^{+}$, the following diagram commutes
\begin{equation}
    \begin{tikzcd}[row sep=huge]
    B_{\mu}(\Lambda)\otimes B_{\nu}(\Lambda)\arrow[r,"g_{\mu,\nu}"], \arrow[d,"\iota^{\mu}_{\mu'}\otimes \iota^{\nu}_{\nu'}"] & B_{\mu+\nu}(\Lambda) \arrow[d,"\iota_{\mu'+\nu'}^{\mu+\nu}"]\\
    B_{\mu'}(\Lambda)\otimes B_{\nu'}(\Lambda)\arrow[r,"g_{\mu',\nu'}"] & B_{\mu'+\nu'}(\Lambda).
    \end{tikzcd}
\end{equation}
By the constructions of $g$'s and $\iota$'s, it suffices to check the commutativity for 

\begin{equation}
    \begin{tikzcd}[row sep=huge]
    P_{\mu'+\nu'}(\Lambda) \arrow[r,"f_{\mu',\nu'}"] \arrow[d,"p^{\mu'+\nu'}_{\mu+\nu}"] & P_{\mu'}(\Lambda)\star P_{\nu'}(\Lambda) \arrow[d,"p^{\mu'}_{\mu}\star p^{\nu'}_{\nu}"] \\
     P_{\mu+\nu}(\Lambda) \arrow[r,"f_{\mu,\nu}"] & P_{\mu}(\Lambda)\star P_{\nu}(\Lambda),
    \end{tikzcd}
\end{equation}
Here, maps $p_{\bullet}^{\bullet}$ appear in the above diagram are the maps $p_{\bullet}^{\bullet}$ in Proposition $5.3.(1)$. The construction of $f$'s implies we are left to show the following diagram commutes
$$
\begin{tikzcd}[row sep=huge]
\textnormal{H}^{*}(P_{\mu'}(\Lambda)\star P_{\nu'}(\Lambda) ) \arrow[d,"\textnormal{H}^{*}(p^{\mu'}_{\mu}\star p_{\nu'}^{\nu})"]\arrow[r,"\cong"] & \textnormal{H}^{*}(P_{\mu'}(\Lambda))\otimes \textnormal{H}^{*}(P_{\nu'}(\Lambda)) \arrow[d,"\textnormal{H}^{*}(p_{\mu}^{\mu'})\otimes \textnormal{H}^{*}(p_{\nu}^{\nu'})"]\\
\textnormal{H}^{*}(P_{\mu}(\Lambda)\star P_{\nu}(\Lambda)) \arrow[r,"\cong"] & \textnormal{H}^{*}(P_{\mu}(\Lambda))\otimes \textnormal{H}^{*}(P_{\nu}(\Lambda)).
\end{tikzcd}
$$
  By the monoidal structrure of $\textnormal{H}^{*}$, the above diagram commutes and so is diagram $(9.2)$. Taking direct limit, the morphisms $g_{\mu,\nu}$ give a multiplication map on $B(\Lambda)$ by the above discussion. Our observation at the end of section $3$ ensures the multiplication on $B(\Lambda)$ is associative. 
\\

Note clearly, $B_0(\Lambda)=\Lambda$ and the canonical map $B_0(\Lambda)\rightarrow B(\Lambda)$ gives the unit map for $B(\Lambda)$. Now to endow $B(\Lambda)$ with a bialgebra structure in the sense of \cite[\S 2]{DM}, it suffices to prove the multiplication on $B(\Lambda)$ is commutative and $B(\Lambda)$ admits an antipode. The later statement can be proved in a completely similar manner as in \cite[Proposition 13.4]{BR} once the former statement is proved. Thus it suffices to prove the commutativity of the mulitiplication of $B(\Lambda)$ . 
\\

First by the compatibility of morphisms $g_{\bullet,\bullet}$ with $\iota_{\bullet}^{\bullet}$, it suffices to prove for each $\mu\in\mathbb{X}_{\bullet}^{+}$, the multiplication on $B_{\mu}(\Lambda)$ is commutative. Consider the following diagram
\begin{equation}
    \begin{tikzcd}[row sep=huge]
    B_{\mu}(\Lambda)\otimes B_{\mu}(\Lambda) \arrow[r,"d"] \arrow[d,hook] & B_{2\mu}(\Lambda) \arrow[d,hook]\\
    B_{\mu}(\bar{\mathbb{Q}}_{\ell})\otimes B_{\mu}(\bar{\mathbb{Q}}_{\ell}) \arrow[r,"d^{'}"] & B_{2\mu}(\bar{\mathbb{Q}}_{\ell})
    \end{tikzcd}
\end{equation}
The vertical arrows are inclusions by noting $(7.1)$ and the fact that $B_{\mu}(\Lambda)$ is a finite free $\Lambda$-module. The map $d$ is defined to map $b_1\otimes b_2$ to $g_{\mu,\mu}(b_1\otimes b_2)-g_{\mu,\mu}(b_2\otimes b_1)$. The map $d'$ is defined similarly. By Proposition $6.6$ and isomorphism $(7.1)$, diagram $(7.4)$ is commutative. By the construction of the commutativity constraint in $P_{L^+G}(Gr_G,\bar{\mathbb{Q}}_{\ell})$ in \cite{Zh2}, we conclude that $d$ is the zero map and the multiplication map in $B(\Lambda)$ is thus commutative. Thus, by a complete similar argument as in \cite[Proposition 2.16]{DM}, the category $\textnormal{Comod}_{B(\Lambda)}$ can be equipped with a commutativity constraint. This commutativity constraints then induces that of $\textnormal{Sat}_{G,\Lambda}$. Thus we have endowed $\textnormal{Sat}_{G,\Lambda}$ with a tensor category structure. Then a complete similar argument as in \cite[Proposition 13.4]{BR} shows that $B(\Lambda)$ admits an antipode.

\section{Identification of Group Schemes}
With the work in previous sections, we have constructed the category $P_{L^+G}(Gr_G,\Lambda)$, and equiped it with 
\begin{itemize}
    \item [(1)] the convolution product $\star$ and an associativity constraint,
    \item [(2)] the hypercohomology functor $\textnormal{H}^{*}:P_{L^+G}(Gr_G,\Lambda)\rightarrow \textnormal{Mod}_{\Lambda}$ which is $\Lambda$-linear, exact and faithful,
    \item [(3)] a commutativity constraint which makes $\textnormal{Sat}_{G,\Lambda}$ a tensor category,
    \item[(3)] a unit object $\textnormal{IC}_0$,
    \item[(4)] a bialgebra $B(\Lambda)$ such that $\textnormal{Sat}_{G,\Lambda}$ is equivalent to $\textnormal{Comod}_{B(\Lambda)}$ as tensor categories.
\end{itemize}
Note by Proposition $5.3$ $(2)$, $\textnormal{H}^{*}(P_Z(\mathbb{Z}_{\ell}))$ is a free $\mathbb{Z}_{\ell}$-module for any $Z\subset Gr_G$ consisting of a finite union of $L^+G$-orbits $Z$. We also know the representing object $P_Z(\mathbb{Z}_{\ell})$ is stable under base change by Proposition $5.3$ $(3)$. By our discussion in the previous section, we have following generalized Tannakian construction similar to \cite[Proposition 11.1]{MV2}
\begin{proposition}
The category
of representations of the group scheme $\widetilde{G}_{\mathbb{Z}_{\ell}}:=\textnormal{Spec}(B(\mathbb{Z}_{\ell}))$ which are finitely generated over $\mathbb{Z}_{\ell}$, is equivalent to $P_{L^+G}(Gr_G,\mathbb{Z}_{\ell})$ as tensor categories. Furthermore,
the coordinate ring of $\widetilde{G}_{\mathbb{Z}_{\ell}}$ is free over $\mathbb{Z}_{\ell}$ and $\widetilde{G}_{\mathbb{F}_{\ell}}=\textnormal{Spec}(\mathbb{F}_{\ell})\times_{\mathbb{Z}_{\ell}}\widetilde{G}_{\mathbb{Z}_{\ell}}$.
\end{proposition}
We are left to identify the group scheme $\widetilde{G}_{\mathbb{Z}_{\ell}}$ with the Langlands dual group $\hat{G}_{\mathbb{Z}_{\ell}}$.

Note reductive group schemes over $\mathbb{Z}_{\ell}$ are uniquely determined by their root datum. Then it suffices to prove the followings for our purpose
\begin{itemize}
    \item [(1)] $\widetilde{G}_{\mathbb{Z}_{\ell}}$ is smooth over $\mathbb{Z}_{\ell}$,
    \item[(2)] the group scheme $\widetilde{G}_{\bar{\mathbb{F}}_{\ell}}$ is reductive,
    \item[(3)] the dual split torus $\hat{T}_{\mathbb{Z}_{\ell}}$ is a maximal torus of $\widetilde{G}_{\mathbb{Z}_{\ell}}$.
\end{itemize}
In \S $7$, we showed that $B(\mathbb{Z}_{\ell})$ is a free $\mathbb{Z}_{\ell}$-modules. As a result, the group scheme $\widetilde{G}_{\mathbb{Z}_{\ell}}$ is affine flat over $\mathbb{Z}_{\ell}$. Then the affineness of $\widetilde{G}_{\mathbb{Z}_{\ell}}$ together with the statements $(1)$ and $(2)$ in this paragraph amount to the definition of a reductive group over $\mathbb{Z}_{\ell}$. Recall in \cite{PY}, a group scheme $\mathcal{G}$ over a discrete valuation ring $R$ with uniformizer $\pi$, field of fractions $K$, and residue field $\kappa$ is said to be \textit{quasi-reductive} if 
\begin{itemize}
    \item [(1)]$\mathcal{G}$ is affine flat over $R$,
    \item[(2)] $\mathcal{G}_{K}:=\mathcal{G}\otimes_R K$ is connected and smooth over $K$,
    \item[(3)] $\mathcal{G}_{\kappa}:=\mathcal{G}\otimes_R\kappa$ is of finite type over $\kappa$ and the neutral component $(\mathcal{G}_{\bar{\kappa}})^{\circ}_{\textnormal{red}}$ of the reduced geometric fibre is a reductive group of dimension equals $\textnormal{dim}\mathcal{G}_{K}$.
\end{itemize}
We will make use of the following theorem for quasi-reductive group schemes proved in \cite{PY}.
\begin{thm}
Let $\mathcal{G}$ be a quasi-reductive group scheme over $R$. Then
\begin{itemize}
    \item [(1)] $\mathcal{G}$ is of finite type over $R$
    \item[(2)] $\mathcal{G}_K$ is reductive
    \item[(3)] $\mathcal{G}_{\kappa}$ is connected.
\end{itemize}
In addition if
\begin{itemize}
    \item [(4)] \text{the type of }$\mathcal{G}_{\bar{K}}$ \text{ is of the same type as that of } $(\mathcal{G}_{\bar{\kappa}})^{\circ}_{\textnormal{red}}$,
\end{itemize}
then $\mathcal{G}$ is reductive.
.
\end{thm}
As noted above, the requirement $(1)$ of quasi-reductiveness is satisfied by $\widetilde{G}_{\mathbb{Z}_{\ell}}$. In addition, by \cite{Zh2}, the group scheme $\widetilde{G}_{\mathbb{Q}_{\ell}}$ is connected reductive with root datum dual to that of $G$ and the condition $2$ of quasi-reductiveness is met.
\begin{lemma}
The group scheme $\widetilde{G}_{\bar{\mathbb{F}}_{\ell}}$ is connected.
\end{lemma}
\begin{proof}
Note the same proof as in \cite[\S $12$]{MV2} and \cite[Lemma $9.3$]{BR} applies in our setting to show the Satake category $P_{L^+G}(Gr_G,\bar{\mathbb{F}}_{\ell})$ has no object $\mathcal{F}$ such that the subcategory $\left<\mathcal{F}\right>$, which is the strictly full subcategory of $P_{L^+G}(Gr_G,\bar{\mathbb{F}}_{\ell})$ whose objects are those isomorphic to a subquotient
of $\mathcal{F}^{\star n}$ for some $n\in \mathbb{N}$, is stable under $\star$. This is equivalent to the fact that there doesn't exist an object $X\in \textnormal{Rep}_{\bar{\mathbb{F}}_{\ell}}(\widetilde{G}_{\bar{\mathbb{F}}_{\ell}})$ such that $\left< X\right> $ is stable under $\bigotimes$ via Proposition $8.1$. Then by \cite[Corollary 2.11.2]{BR}, we conclude our proof.
\end{proof}

From now on, let $\kappa=\bar{\mathbb{F}}_{\ell}$. We have proved in Proposition $6.4$ that the monoidal structure of $\textnormal{H}^{*}$ is compatible with the weight functor decomposition. In other words, we get a monoidal functor 
$$
\textnormal{CT}:\textnormal{Sat}_{G,\mathbb{Z}_{\ell}}\longrightarrow\textnormal{Mod}_{\mathbb{Z}_{\ell}}(\mathbb{X}_{\bullet})\simeq \textnormal{Sat}_{T,\mathbb{Z}_{\ell}}.
$$
Base change to $\kappa$, the same reasoning yields a monoidal functor
$$
\textnormal{CT}:\textnormal{Sat}_{G,\kappa}\longrightarrow\textnormal{Mod}_{\kappa}(\mathbb{X}_{\bullet})\simeq \textnormal{Sat}_{T,\kappa}.
$$
Applying the construction in \S 7 to the above two Satake categories, we get a natural homomorphism $\hat{T}\rightarrow \widetilde{G}$. Note that by \cite[Corollary 2.8]{Zh2} and Proposition $5.3.(2)$, any $M\in \textnormal{Mod}_{\kappa}(\mathbb{X}_{\bullet})$ can be realized as a subquotient of some projective object in $\textnormal{Sat}_{G,\kappa}$. It then follows from \cite[Proposition 2.21(b)]{DM} that the homomorphism $\hat{T}\rightarrow \widetilde{G}$ is in fact a closed embedding, which realizes the dual torus $\hat{T}_{\kappa}$ as a subtorus of $\widetilde{G}_{\kappa}$. In addition, since $\widetilde{G}_{\mathbb{Z}_{\ell}}$ is flat, the same argument in \cite[\S 12]{MV2} applies to give the following dimension estimate
\begin{equation}
    \dim G=\dim \widetilde{G}_{\mathbb{Q}_{\ell}}\geq \dim (\widetilde{G}_{\kappa})_{\textnormal{red}}.
\end{equation}

We can write $\widetilde{G}_{\kappa}=\varprojlim \widetilde{G}^*_{\kappa}$ where  $\widetilde{G}^*_{\kappa}$ satisfies the following conditions
\begin{itemize}
    \item [(1)] $\widetilde{G}^*_{\kappa}$ is of finite type,
    \item [(2)] the canonical map $\textnormal{Irr}_{\widetilde{G}^*_{\kappa}}\rightarrow \textnormal{Irr}_{\widetilde{G}_{\kappa}}$ is a bijection, where Irr denotes the set of irreducible representations.
\end{itemize}
In addition, we require that the transition morphisms are surjective. The first requirement may be satisfied since any group
scheme is a projective limit of group schemes of finite type. To ensure that condition $(2)$ can be satisfied,
it is enough to choose $\widetilde{G}^*_{\kappa}$ sufficiently large so that the irreducible representations $L(\eta)$ associated to a finite set of generators $\eta$ of the semigroup of dominant cocharacters $\mathbb{X}_{\bullet}^+$, are pull-backs of representations of $\widetilde{G}^*_{\kappa}$. For any $\mu,\nu\in\mathbb{X}_{\bullet}^+$, the sheaf $\textnormal{IC}_{\mu+\nu}$ supports on $Gr_{\leq\mu+\nu}$ and hence is a subquotient of $\textnormal{IC}_{\mu}\star\textnormal{IC}_{\nu}$. Thus all irreducible representations of $\widetilde{G}_{\kappa}$ come from $\widetilde{G}_{\kappa}^*$. By our choice of the finite type quotients, we have $(\widetilde{G}_{\kappa})_{\textnormal{red}}=\varprojlim (\widetilde{G}_{\kappa}^*)_{\textnormal{red}}$. In addition, the composition of maps $\hat{T}_{\kappa}\rightarrow \widetilde{G}_{\kappa}\rightarrow \widetilde{G}^*_{\kappa}$ is a closed embedding. 
\\

We claim that
\begin{equation}
\
    \text{Each finite type quotient } \widetilde{G}_{\kappa}^* \text{ is connected, reductive, and isomorphic to } \hat{G}_{\kappa}.
\end{equation}
If $(8.2)$ holds, then the arguments in \cite{MV3} apply and yield $\widetilde{G}_{\kappa}^*= (\hat{G})_{\kappa}$. Thus we deduce that condition $(3)$ of the quasi-reductiveness and condition $(4)$ in Theorem $8.2$ are satisfied by $\widetilde{G}_{\kappa}$ and we complete the identification of group schemes by Theorem $8.2$. Next, we prove $(8.2)$ following the appraoch given in \cite[\S 12]{MV2}.
\\

Write $H$ for the reductive quotient of $(\widetilde{G}_{\kappa}^*)_{\textnormal{red}}$ and we have $\hat{T}_{\kappa}\rightarrow H$ is a closed embedding.
Note any irreducible representation of $(\widetilde{G}_{\kappa}^*)_{\textnormal{red}}$ is trivial on the unipotent radical, we have:
\begin{equation}
    \text{The canonical map }\textnormal{Irr}_H\rightarrow \textnormal{Irr}_{(\widetilde{G}_{\kappa}^*)_{\textnormal{red}}}\text{ is a bijection.}
\end{equation}
We first note the following lemma.
\begin{lemma}
The subtorus $\hat{T}_{\kappa}$ is a maximal subtorus of $H$.
\end{lemma}

 \begin{proof}

Choose a maximal torus $T_H$ for $H$ and denote its Weyl group $W_H$. Then the irreducible representations of $H$ are parametrized by $\mathbb{X}^{\bullet}(T_H)/W_H$. On the other hand, write the Weyl group for $G$ by $W_G$, then Proposition $2.3$ implies that $\mathbb{X}_{\bullet}(T_{\kappa})/W_G$ parametrizes Schubert cells in $Gr_{G_{\kappa}}$. The IC-sheaf attached to each Schubert cell is an irreducible object in the Satake category and thus gives rise to an irreducible representation of $\widetilde{G}_{\kappa}$. By our choice of $\widetilde{G}^*_{\kappa}$ and $(8.3)$, we get a bijection $\mathbb{X}^{\bullet}(T_H)/W_H\simeq \mathbb{X}_{\bullet}(T)/W_G$. Hence, $T_H/W_H\simeq \hat{T}_{\kappa}/W_G$. Note that the Weyl group acts faithfully on the maximal torus, we conclude that $\mathbb{X}^{\bullet}(T_H)=\mathbb{X}_{\bullet}(T)$ and $\hat{T}_{\kappa}$ is a maximal torus in $H$.

 \end{proof}
From now on, we write $W_H$ for the Weyl group of $H$ with respect to $\hat{T}_{\kappa}$. Recall a (co)character of a reductive group is called regular if the cardinality of its orbit under the Weyl group action attains the maximum. Then $2\rho$ is a regular character in $G$ with respect to $T$. By the proof of Lemma $8.4$, it is a cocharacter in $H$ with respect to $\hat{T}_{\kappa}$. In addition, the proof of Lemma $8.4$ also shows that $W_H\cdot 2\rho=W_G\cdot 2\rho$ and thus the Weyl group orbit $W_H\cdot 2\rho$ has maximal cardinality and it follows that $2\rho$ is a regular cocharacter in $H$. Thus $2\rho$ fixes a Borel $B_H$ which only depends on the Weyl chamber containing $2\rho$. It also fixes a set of positive roots.  
\\

From the proof of Lemma $8.4$, we deduce the followings
\begin{equation}
    \text{the (dominant)weights of } (H,B_H,\hat{T}_{\kappa})\text{ coincide with }\text{(dominant) coweights of }G.
\end{equation}
\begin{equation}
   W_H \text{ coincides with } W_G \text{ together with their subsets of simple reflections identified}.
\end{equation}

To show $(8.2)$, we hope to prove the following  
\begin{equation}
    \Delta (H,B_H,\hat{T}_{\kappa})=\Delta^{\vee}(G,B,T)
\text{ and } \Delta^{\vee} (H,B_H,\hat{T}_{\kappa})=\Delta (G,B,T).
\end{equation}
We first prove a weaker version of the first equality in $(8.6)$.

\begin{lemma}
Assume $G$ to be semisimple, then statement $(8.6)$ holds.
\end{lemma}
\begin{proof}
Since $G$ is assumed to be semisimple, then $\mathbb{Q}\cdot\mathbb{X}_{\bullet}^+(T)=\mathbb{Q}\cdot \Delta^{\vee}(G,B,T)$. Hence, 
\begin{equation}
\mathbb{Z}_{\geq 0}\cdot\Delta_s(G,B,T)=\{\alpha\in \mathbb{X}^{\bullet}(T)\vert \big<\alpha,\lambda\big>\geq 0\text{ for all }\lambda\in \mathbb{X}_{\bullet}^+(T)\}. 
\end{equation}
On the other hand, it follows from $(8.5)$ that $W_H$ and $W_G$ have the same cardinality. Together with $(8.4)$, we conclude that $H$ is also semisimple. Thus, 
\begin{equation}
\mathbb{Z}_{\geq 0}\cdot\Delta^{\vee}_s(H,B_H,\hat{T}_{\kappa})=\{\alpha^{\vee}\in \mathbb{X}_{\bullet}(\hat{T}_{\kappa})\vert \big<\alpha,\lambda\big>\geq 0\text{ for all }\lambda\in \mathbb{X}^{\bullet}_+(\hat{T}_{\kappa})\}. 
\end{equation}
Comparing $(8.7)$ and $(8.8)$, we have 
$$
\mathbb{Z}_{\geq 0}\cdot\Delta_s(G,B,T)=\mathbb{Z}_{\geq 0}\cdot\Delta^{\vee}_s(H,B_H,\hat{T}_{\kappa}).
$$
Thus, $\Delta_s(G,B,T)=\Delta^{\vee}_s(H,B_H,\hat{T}_{\kappa})$ and we conclude that $\Delta(G,B,T)=\Delta^{\vee}(H,B_H,\hat{T}_{\kappa})$ by noting $(8.5)$. Finally, since for a semisimple reductive group, the coroots are uniquely determined by roots and vice versa, we also conlcude that $\Delta^{\vee}(G,B,T)=\Delta(H,B_H,\hat{T}_{\kappa})$.

\end{proof}
In fact, Lemma $8.5$ may also be proved following the idea of \cite[\S 14]{BR}\footnote{We note that the situation considered in \cite[\S 14]{BR} is slightly different from ours. In the equal characteristic situation considered in \textit{loc.cot}, the group scheme $\widetilde{G}_{\kappa}$ is proved to be algebraic by directly exhibiting a tensor generator of the Satake category. Then, there is no need to pass to finite type quotient $\widetilde{G}_{\kappa}^{*}$ as we do in this section. } and \cite[\S 12]{MV2} and we sketch this approach here.  
\begin{lemma}
We have the following inclusion of lattices
\begin{equation}
    \mathbb{Z}\cdot \Delta (H,\hat{T}_{\kappa})\subseteq \mathbb{Z}\cdot \Delta^{\vee}(G,T)
\end{equation}
for general $G$.
\end{lemma}
\begin{proof}
The proof is similar to that for \cite[(12.21)]{MV2} and we sketch it here. Note the Satake category $\textnormal{Sat}_{G,\kappa}$ is equipped with a grading by $\pi_0(Gr_G)\simeq \pi_1(G)=\mathbb{X}_{\bullet}(T)/\mathbb{Z}\cdot \Delta^{\vee}(G,T)$ by \cite[Proposition 1.21]{Zh2}. In addition, this grading is compatible with the tensor structure in $\textnormal{Sat}_{G,\kappa}$. Write $Z$ for the center of $G$, then it can be identified with the group scheme
\begin{equation}
    \textnormal{Hom}(\mathbb{X}_{\bullet}(T)/\mathbb{Z}\cdot \Delta^{\vee}(G,T),\mathbb{G}_{m,\kappa}).
\end{equation}
Our previous observation implies that the forgetful functor 
$$
\textnormal{Sat}_{G,\kappa}\simeq \textnormal{Rep}_{\kappa}(\widetilde{G}_{\kappa})\longrightarrow \textnormal{Rep}_{\kappa}(Z)
$$
is compatible with the grading considered above. In this way $Z$ is realized as a central subgroup of $\widetilde{G}_{\kappa}$. Since $\hat{T}_{\kappa}\rightarrow H$ is a closed embedding, $Z$ is also contained in the center of $H$. Finally, note the center of $H$ can be identified with the group scheme
\begin{equation}
    \textnormal{Hom}(\mathbb{X}^{\bullet}(\hat{T}_{\kappa})/\mathbb{Z}\cdot \Delta^{\vee}(H,\hat{T}_{\kappa}),\mathbb{G}_{m,\kappa}).
\end{equation}
Our discussion together with $(8.10)$ and $(8.11)$ completes the proof of the lemma.
\end{proof}
\begin{lemma}
The set of dominant weights of $(H,\hat{T}_{\kappa})$ is equal to $\mathbb{X}_{\bullet}^+(T)\subset \mathbb{X}_{\bullet}(T)=\mathbb{X}^{\bullet}(\hat{T}_{\kappa})$.
\end{lemma}
\begin{proof}
By our construction, we have a bijection between the set of irreducible representations of $\widetilde{G}_{\kappa}$ and that of $\widetilde{G}_{\kappa}^{*}$. Since irreducible representations restrict trivially to the unipotent radical, we get a bijection between the set of irreducible representations of $\widetilde{G}_{\kappa}$ and that of $H$. Thus, the dominant weights of $(H,\hat{T}_{\kappa})$ equals that of $(\widetilde{G}_{\kappa},\hat{T}_{\kappa})$.
\\

Let $\lambda\in\mathbb{X}_{\bullet}(T)$ be a dominant weight of $(\widetilde{G}_{\kappa},\hat{T}_{\kappa})$ and write the $L^{\widetilde{G}_{\kappa}}(\lambda)$ for the irreducible representation of $\widetilde{G}_{\kappa}$ associated to $\lambda$. Assume $\mu\in \mathbb{X}_{\bullet}^+(T)$ be a dominant coweight of $G$ such that the simple perverse sheaf corresponding to $L(\lambda)$ is $\textnormal{IC}_{\mu}$. Note in the Grothendieck group of $\textnormal{Sat}_{G,\kappa}$, we have
$$
[\textnormal{IC}_{\mu}]=\Big [{^{p}}j_{\mu,*}\underline{\kappa}_{Gr_{\mu}}[(2\rho,\mu)]\Big ]+\sum_{\nu\in\mathbb{X}_{\bullet}^+(T),\nu<\mu} a_{\nu}^{\mu}\Big [{^{p}}j_{\nu,!}\underline{\kappa}_{Gr_{\nu}}[(2\rho,\nu)]\Big ].
$$
Then we conclude that $\lambda=\mu\in\mathbb{X}_{\bullet}^+(T)$. 
\\

On the other hand, if $\mu\in\mathbb{X}_{\bullet}^+(T)$, then the weights of the $\widetilde{G}_{\kappa}$-representation which corresponds to $^{p}j_{\mu,!}\underline{\kappa}_{Gr_{\mu}}$ are independent of the coefficient $\kappa$ by \cite[Proposition.  8.1]{MV2}. Hence, they are weights of the irreducible $\hat{G}_{\bar{\mathbb{Q}}_{\ell}}$-representation of highest weight $\mu$. Thus $\mu$ is a dominant weight of $(\widetilde{G}_{\kappa},\hat{T}_{\kappa})$. 
\end{proof}

\begin{lemma}
The Weyl groups $W_G$ and $W_H$ coincide when considered as automorphism groups of $\mathbb{X}_{\bullet}(T)$, together with their respective subsets of simple reflections $S_G$ and $S_H$, coincide.
\end{lemma}
\begin{proof}
The proof of this lemma is completely similar to the proof of \cite[Lemma 14.9]{BR} and we sketch it here. For any $\lambda\in\mathbb{X}_{\bullet}^+(T)$, we consider it as a dominant weight of $(H,\hat{T}_{\kappa})$. Then the orbit $W_H\cdot \lambda$ is the set of extremal points of the convex polytope consisting of the convex hull of weights of the irreducible $H$-representation $L^H(\lambda)$. Since the set of irreducible representations of $H$ are bijective to that of $\widetilde{G}_{\kappa}$, we conclude that 
\begin{equation}
    W_H\cdot \lambda=W_G\cdot \lambda.
\end{equation}

For any $\lambda\in\mathbb{X}_{\bullet}^+(T)$, we call $\lambda$ regular if it is not orthogonal to any simple root of $(G,T)$. Then for a regular $\lambda\in \mathbb{X}_{\bullet}^+(T)$, the orbit $S_G\cdot \lambda\subset W\cdot \lambda$ is the subset of $W_G\cdot\lambda$ consisting of elements $\mu$ such that the line segment connecting $\lambda$ and $\mu$ is extremal in the convex hull of $W_G\cdot \lambda$. By $(8.12)$, we have the same description for the orbit $S_H\cdot \lambda$. Thus, 
\begin{equation}
    S_G\cdot\lambda=S_H\cdot \lambda.
\end{equation}

Choose an arbitrary $s_G\in S_G$. For any $\lambda\in \mathbb{X}_{\bullet}^+(T)$ regular, by $(8.13)$ there exists $s_H\in S_H$ such that $s_G\cdot \lambda=s_H\cdot \lambda$. In addition, the direction of the line segment connecting $\lambda$ with $s_G\cdot \lambda$ is determined by the line segment joining the coroot of $G$ associated with $s_G$ with the root of $H$ associated with $s_H$. Thus for any any other $\lambda'\in \mathbb{X}_{\bullet}^+(T)$ regular, we also have $s_G\cdot\lambda'=s_H\cdot \lambda'$. It follows that $s_G=s_H$ and thus $S_G=S_H$. Thus we deduce that $W_G=W_H$.
\end{proof}
\begin{lemma}
We have the following inclusion of lattices
$$
\mathbb{Z}\cdot \Delta (G,T)\subseteq \mathbb{Z}\cdot \Delta^{\vee}(H,\hat{T}_{\kappa}). 
$$
\end{lemma}
\begin{proof}
The proof is similar to the one for \cite[Lemma 14.10]{BR} and we sketch it here. Firstly, we observe by Lemma $8.7$ that 
\begin{equation}
\mathbb{Q}_{+}\cdot \Delta_s^{\vee}(H,B_H,\hat{T}_{\kappa})=\mathbb{Q}_{+}\cdot \Delta_s(G,B,T).
\end{equation}
This is because both sets consist of extremal rays of the rational convex polyhedral cone determined by $\{\lambda\in\mathbb{Q}\otimes_{\mathbb{Z}}\mathbb{X}^{\bullet}(T)\vert \textnormal{ for any }\mu\in \mathbb{X}_{\bullet}^+(T),\big<\lambda,\mu\big>\geq 0\}$. For $\mu\in \Delta_s(G,B,T)$, it follows from $(8.14)$ that there exists $a\in \mathbb{Q}_+\backslash \{0\}$ such that $a\mu\in \Delta_s^{\vee}(H,B_H,\hat{T}_{\kappa})$. Lemma $8.8$ then implies that
$$
\textnormal{id}-\big<\mu^{\vee},\bullet\big>=\textnormal{id}-\big<(a\mu)^{\vee},\bullet\big>(a\mu)
$$
as an automorphism of $\mathbb{X}^{\bullet}(T)=\mathbb{X}_{\bullet}(\hat{T}_{\kappa})$. Thus, $(a\mu)^{\vee}=\frac{1}{a}\mu^{\vee}$. Note Lemma $8.6$ shows that $a\mu\in \mathbb{Z}\cdot \Delta^{\vee}(G,T)$. Thus, $\frac{1}{a}\in\mathbb{Z}$ and $\mu=\frac{1}{a}(a\mu)\in\mathbb{Z}\cdot \Delta^{\vee}(H,\hat{T}_{\kappa})$.
\end{proof}
The arguments above prepare us for a second proof of Lemma $8.5$ as follows.
\begin{proof}
If $G$ is in particular semisimple of adjoint type, then $\mathbb{Z}\cdot \Delta(G,T)=\mathbb{X}^*(T)$. Lemma $8.9$ then implies that $\mathbb{Z}\cdot \Delta(G,T)=\mathbb{Z}\cdot \Delta^{\vee}(H,\hat{T}_{\kappa})$. Then the arguments in the proof of Lemma $8.9$ imply that $\Delta_s(G,T)=\Delta^{\vee}_s(H,\hat{T}_{\kappa})$. In addition, $\Delta_s^{\vee}(G,B,T)=\Delta_s(H,B_H,\hat{T}_{\kappa})$ and the canonical bijections between the roots and coroots of $H$ and $G$ coincide. It then follows from Lemma $8.8$ that $\Delta(H,B_H,\hat{T}_{\kappa})=\Delta^{\vee}(G,T)$ and $\Delta^{\vee}(H,\hat{T}_{\kappa})=\Delta(G,T)$. Thus the root datum of $H$ with respect to $\hat{T}_{\kappa}$ is dual to that of $(G,T)$. Then the dimension estimate $(8.1)$ concludes the proof of the lemma in the semisimple of adjoint type case.
\\

Assume $G$ is a general semisimple reductive group scheme. We denote by $G_{\textnormal{ad}}$ the adjoint quotient of $G$ and $T_{\textnormal{ad}}$ the quotient of the maximal torus $T$. The construction in \S7 goes through and we get the group scheme $(\widetilde{G}_{\textnormal{ad}})_{\kappa}$. As noted in the proof of Lemma $8.6$, the Satake category $\textnormal{Sat}_{G_{\textnormal{ad}},\kappa}$ admits a grading  by the finite group $\pi_1(G_{\textnormal{ad}})/\pi_{1}(G)$ which is compatible with the tensor structure of $\textnormal{Sat}_{G_{\textnormal{ad}},\kappa}$. By Lemma $4.6$, the category $\textnormal{Sat}_{G,\kappa}$ can be realized as a tensor subcategory of $\textnormal{Sat}_{G_{\textnormal{ad}},\kappa}$ corresponding to the identity coset of $\pi_1(G)$. Thus, we have a surjective quotient $$
(\widetilde{G})_{\textnormal{ad},\kappa}\twoheadrightarrow \widetilde{G}_{\kappa}
$$
with finite central kernel given by $\textnormal{Hom}(\pi_1(G_{\textnormal{ad}})/\pi_{1}(G),\mathbb{G}_{m,\kappa})$. Hence, $\widetilde{G}_{\kappa}$ is reductive and in particular semisimple. The result for $G$ being semisimple of adjoint type applies here to complete the proof. 
\end{proof}
Now, we complete the final step of identifying the group schemes.

\begin{lemma}
Let $G$ be a general connected reductive group, then the same result as in Lemma $8.6$ holds. 
\end{lemma}
\begin{proof}
We sketch a proof similar to the arguments for \cite[\S 12]{MV2} and \cite[Lemma 14.13]{BR}. Denote  by $Z(G)$ the center of $G$ and let $A=Z(G)^{\circ}$. Then $A$ is a torus and $G/A$ is semisimple. As in \textit{loc.cit}, the exact sequence
$$
1\rightarrow A\rightarrow G\rightarrow G/A\longrightarrow 1
$$
induces maps
$$
Gr_A\xrightarrow{i} Gr_{G}\xrightarrow{\pi} Gr_{G/A}
$$
which exhibit $Gr_{G}$ as a trivial $Gr_A$-cover over $Gr_{G/A}$. This induces an exact sequence of functors
\begin{equation}
P_{L^+A}(Gr_A,\kappa)\xrightarrow{i_*} P_{L^+G}(Gr_{G},\kappa) \xrightarrow{\pi_*} P_{L^+G/A}(Gr_{G/A},\kappa).
\end{equation}
Note that $(Gr_A)_{\textnormal{red}}$ is a set of discrete points indexed by $\mathbb{X}_{\bullet}^+(A)$, then taking pushforward along $i$ gives a fully faithful functor $i_*: P_{L^+A}(Gr_A,\kappa)\rightarrow P_{L^+G}(Gr_{G},\kappa)$. The functor $\pi_*$ is made sense by Lemma $4.6$ and is essentially surjective.
\\

Applying the Tannakian construction as in \S 7, we get flat affine group schemes $\widetilde{A}_{\kappa}$ and $\widetilde{(G/A)}_{\kappa}$. Lemma $8.5$ implies that $\widetilde{A}_{\kappa}$ and $\widetilde{(G/A)}_{\kappa}$ are isomorphic to the dual groups of $H$ and $G/A$ respectively. The same arguments in \cite[\S 12]{MV2} and \cite[\S 14]{BR} apply here to deduce that the sequence 
$$
1\longrightarrow \widetilde{G/A}_{\kappa}\longrightarrow \widetilde{G}_{\kappa}\longrightarrow \widetilde{A}_{\kappa}\longrightarrow 1
$$
induced by $(8.9)$ is exact. Then $\widetilde{G}_{\kappa}$ is identified as the extension of smooth group schemes $\widetilde{A}_{\kappa}$ and $\widetilde{G/A}_{\kappa}$ and is thus also smooth. Moreover, the unipotent radical of $\widetilde{G}_{\kappa}$ has trivial image in the torus $\widetilde{A}_{\kappa}$. Hence it is included in $\widetilde{G/A}_{\kappa}$. Since
the latter group is semisimple it follows that $\widetilde{G}_{\kappa}$ is also reductive. Argue as in \cite[Lemma 14.14]{BR}, we complete the proof of the lemma.
\end{proof}

Thus we identify the group scheme $\widetilde{G}_{\mathbb{Z}_{\ell}}$ which arises from the general Tannkian construction with the Langlands dual group $\hat{G}_{\mathbb{Z}_{\ell}}$. We have established our main theorem.
\begin{thm}
There is an equivalence of tensor categories between $P_{L^+G}(Gr_G,\Lambda)$ and the category of representations of the Langlands dual group $\hat{G}_{\Lambda}$ of $G$ on finitely generated $\Lambda$-modules for $\Lambda=\mathbb{F}_{\ell}$, and $\mathbb{Z}_{\ell}$. 
\end{thm}

\bibliographystyle{amsplain}

\end{document}